\documentclass[10pt]{amsart}
\usepackage{amsmath, amssymb}
\newcommand{\no}[1]{#1}
\renewcommand{\no}[1]{}
\no{\usepackage{times}\usepackage[subscriptcorrection, slantedGreek, nofontinfo]{mtpro}
\renewcommand{\Delta}{\upDelta}}
\usepackage{color}
 \usepackage{mathrsfs}
\newcommand{\mdiv}{{\mathrm{div}}}
\date{\today}
\newtheorem{theorem}{Theorem}[section]
\newtheorem{proposition}{Proposition}[section]
\newtheorem{lemma}{Lemma}[section]

\newtheorem{corollary}{Corollary}[section]

\theoremstyle{remark}
\newtheorem{remark}{Remark}[section]

\numberwithin{equation}{section}

\newlength{\textlarg}

\title[An inverse anisotropic conductivity problem]{An inverse anisotropic conductivity problem induced by twisting a homogeneous cylindrical domain}

\author[Mourad Choulli]{Mourad Choulli\dag}
\address{\dag Institut \'Elie Cartan de Lorraine, UMR CNRS 7502, Universit\'e de Lorraine, Ile du Saulcy, F-57045 Metz cedex 1, France}
\email{mourad.choulli@univ-lorraine.fr}

\author[Eric Soccorsi]{Eric Soccorsi\ddag}
\address{\ddag Aix-Marseille Universit\'e, CNRS, CPT UMR 7332, 13288 Marseille, France \& Universit\'e de Toulon, CNRS, CPT UMR 7332, 83957 La Garde, France}
\email{eric.soccorsi@univ-amu.fr}

\date{}

\begin{document}

\begin{abstract}
We consider the inverse problem of determining the unknown function $\alpha: \mathbb{R} \rightarrow \mathbb{R}$ from the DN map associated with the operator $\mbox{div}(A(x',\alpha (x_3))\nabla \cdot)$ acting in the infinite straight cylindrical waveguide $\Omega =\omega \times \mathbb{R}$, where $\omega$ is a bounded domain of $\mathbb{R}^2$. Here $A=(A_{ij}(x))$, $x=(x',x_3) \in \Omega$, is a matrix-valued metric on $\Omega$ obtained by straightening a twisted waveguide. This inverse anisotropic conductivity problem remains generally open, unless the unknown function $\alpha$ is assumed to be constant. In this case we prove Lipschitz stability in the determination of $\alpha$ from the corresponding DN map. The same result remains valid upon substituting a suitable approximation of the DN map, provided the function $\alpha$ is sufficiently close to some {\it a priori} fixed constant.

\medskip
\noindent {\bf Key words:} Dirichlet Laplacian, twisted infinite cylindrical domain,  inverse anisotropic conductivity problem, DN map, stability estimate.

\medskip
\noindent
{\bf AMS subject classifications:} 35R30.
\end{abstract}

\maketitle

\tableofcontents


\section{Introduction}
In the present paper we consider an inverse conductivity problem in an anisotropic medium arising from the twisting of a homogeneous infinite straight cylindrical waveguide.
Generically the inverse conductivity problem in an arbitrary domain $\Omega \subset \mathbb{R}^n$, $n \geq 1$, is to determine a symmetric, positive definite matrix $A=A(x)$, $x \in \Omega$, representing the conductivity tensor of $\Omega$, from the Dirichlet-to-Neumann (abbreviated to DN in the following) map associated with $A$:
$$ \Lambda_A: u_{| \partial \Omega} \mapsto A \nabla u \cdot \nu. $$
Here $u$ is the solution to the elliptic equation $\mbox{div} (A \nabla u) =0$ in $\Omega$ and $\nu$ denotes the unit outward normal vector to the boundary $\partial \Omega$ of $\Omega$.

\smallskip
If $A$ is isotropic (i.e. $A$ is the identity matrix of $\mathbb{R}^n$ up to some multiplicative scalar -unknown- function of $x \in \Omega$) Sylvester and Uhlmann \cite{SyU} proved that the conductivity is uniquenely determined from the knowledge of $\Lambda_A$. Similar results were derived by Lionheart in \cite{L} upon substituting a suitable given matrix $A_0(x)$ for the identity. For a conductivity of the form $A(x)=A_0(x,\alpha(x))$ where $A_0$ is given and $\alpha$ is an unknown scalar function, Alessandrini and Gaburro \cite{AG1,AG2} obtained uniqueness and stability under the monotonicity assumption 
\begin{equation}
\label{monotone}
\partial_t A_0(x,t) \geq c I,\ x \in \Omega,\ t \in \mathbb{R}.
\end{equation}
Here $c$ is some positive constant and $I$ stands for the identity of $\mathbb{R}^n$. The case $A(x)=A_0(\alpha(x))$ was treated earlier by Alessandrini in \cite{A} under the same kind of monotonicity condition. 

\smallskip
All the above mentioned results were obtained in a bounded domain of $\mathbb{R}^n$. In this paper we will rather consider an infinite cylindrical straight waveguide $\Omega = \omega \times \mathbb{R}$, where $\omega$ is a bounded domain of $\mathbb{R}^2$. To the variable twisting angle $\theta \in C^1(\mathbb{R})$ we then associate the twisted waveguide
\[
\Omega _\theta =\left\{ (R_{\theta (x_3)}x',x_3);\; x'=(x_1,x_2)\in \omega ,\; x_3\in \mathbb{R}\right\},
\]
where $R_\xi$ denotes the rotation in $\mathbb{R}^2$ of angle $\xi\in \mathbb{R}$. Twisted waveguides modeled by $\Omega_\theta$ exhibit interesting propagation properties such as the occurence of propagating modes with phase velocities slower than
those of similar modes in a straight waveguide. This explains why these peculiar structures are at the center of the attention of many theoretical and applied physicists (see e.g. \cite{Kang,DMR,KF,NZG,Sh,Wi,YM}). From a mathematical viewpoint twisted waveguides are the source of challenging spectral and PDE problems, some of them having been extensively studied in the mathematical literature (see e.g. \cite{BKRS,EKK,KK,KS,KZ1,KZ2}). Moreover the cylindrical geometry is well suited to the construction of complex geometrical optics solutions \cite{DKSU,CKS}, which motivates for the analysis of inverse problems occuring in waveguides.

\smallskip 
We consider the following boundary value problem (abbreviated to BVP) for the Laplacian in $\Omega_\theta$:
\begin{equation}\label{1}
\left\{
\begin{array}{ll}
\Delta v(y)=0,\ y \in \Omega_\theta,
\\
v(y)=g(y),\ y \in \partial \Omega_\theta.
\end{array}
\right.
\end{equation}
Upon straightening $\Omega_\theta$, \eqref{1} may be brought into an equivalent BVP stated in $\Omega$. This can be seen by introducing
\[
T_\xi =
\begin{pmatrix}
R_\xi &0
\\
0&1
\end{pmatrix},
\]
putting $u(x)=v(T_{\theta (x_3)}(x',x_3))$, $x=(x',x_3)\in \Omega$, and performing the change of variable $y=T_{\theta (x_3)}(x',x_3)$. We get using direct calculation that $u$ is solution to the following elliptic BVP in the divergence form 
 \begin{equation}\label{2}
\left\{
\begin{array}{ll}
\mdiv \big(A_0(x',\alpha (x_3))\nabla u\big)=0,\ x \in \Omega,
\\
u(x)=f(x),\ x \in \partial \Omega,
\end{array}
\right.
\end{equation}
where $\alpha =\theta'$, $f(x)=g(T_{\theta (x_3)}(x',x_3))$ for $x\in \partial \Omega$, and the matrix $A_0$ is given by
\[
 A_0(x',t)=
 \begin{pmatrix}
1+x_2^2t^2 & -x_2x_1t^2 &-x_2t
\\
-x_2x_1t^2 & 1+x_1^2t^2 & x_1t
\\
-x_2t & x_1t &1
 \end{pmatrix},\;\; x'\in \omega ,\; t\in \mathbb{R}.
\]
At this point it is worth stressing out that the BVP \eqref{2} is stated on the straight waveguide $\Omega$ and not on the twisted waveguide $\Omega_\theta$ itself. Nevertheless we notice that the geometry of $\Omega_\theta$ is expressed in \eqref{2} through the metric $A=(A_{ij}(x))=A_0(x',\alpha(x_3))$.

\smallskip
The problem we examine in this paper is to know whether the unknown function $\alpha$ can be determined from the DN map
\[
\Lambda_\alpha : f \mapsto A\nabla u \cdot \nu .
\]
This is the same kind of inverse anisotropic conductivity problem, but stated here in the unbounded straight waveguide $\Omega$, as the one studied in \cite{A,AG1,AG2,GL} in a bounded domain. However, it turns out that the usual monotonicity assumption (\ref{monotone}) on the conductivity, which is essential to the identification of $A$ from the DN map in this approach, is not verified in our framework (since the matrix $\partial_t A_0(x',t)$ has a negative eigenvalue $|x'|^2t-\sqrt{|x'|^4t^2+|x'|^2}$, see \eqref{ev}). This explains why the inverse problem associated with \eqref{2} is still open for general unknown functions $\alpha \in C({\mathbb R})$. 
However, if $\alpha$ is sufficiently close to some {\it a priori} fixed constant we prove in Theorem \ref{thm-aconstant} that this unknown function may well be determined from the knowledge of some suitable approximation of the DN map (see \eqref{eq-approx}).
Moreover, in the particular case where $\alpha$ is known to be constant so the conductivity matrix $A$ is invariant w.r.t. the infinite variable $x_3$, the original problem 
is equivalent to some inverse anisotropic conductivity problem stated in $\omega$. The corresponding conductivity matrix satisfies a weak monotonicity condition in this case implying Lipschitz stability in the determination of $\alpha$ from $\Lambda_\alpha$ (see Theorem \ref{theorem2}).

The problem under investigation in this paper is a special case of the anisotropic Calder\'on problem, which consists in determining the geometric structure of a smooth Riemanniann manifold with boundary from the Cauchy data of harmonic functions. 
Actually, the DN mapping being invariant under diffeomorphisms which preserve the boundary, it is well known that there is an obstruction to uniqueness in this problem. Nevertheless, in dimensions greater than or equal to three, it was proved by Lee and Uhlmann \cite{LeU} that compact real-analytic manifolds are uniquely defined by the DN map up to diffeomorphisms preserving the boundary. A similar result was obtained in \cite{GS} by Guillarmou and S\`a Barreto for Einstein metrics that are real-analytic in the interior.
The case of non-analytic simple manifolds was treated in \cite{DKSU} by means of Carleman estimates with limiting weights. Recently, Dos Santos Ferreira, Kurylev, Lassas and Salo proved in \cite{DKLS} that the boundary measurements in an infinite cylinder uniquely determine the transversal metric. 
We point out that in dimension two, there is an additional obstruction to uniqueness for the anisotropic Calder\'on problem, arising from the conformal invariance of the associated Laplace-Beltrami operator. 
See \cite{SU, S, LU} and references therein for a detailed study of this case, and \cite{ALP} for a specific treatment of the same inverse problem in the plane.
Finally, we refer to \cite{IUY1, IUY2} for inverse anisotropic conductivity problems with partial Cauchy data.

The paper is organized as follows. Section 2 gathers the definition and the main properties of $\Lambda_\alpha$. In the first part of section 3 we adapt the method developed by Alessandrini and Gaburro in \cite{AG1, AG2} to determine the conductivity in a bounded anisotropic medium from the DN map, to the particular framework of an infinite cylindrical waveguide under consideration in this text. Although this technique does not allow for the identification of general unknown functions $\alpha$, we prove stability in the determination of $\alpha$ from some suitable approximation of the DN map, provided $\alpha$ is sufficiently close to some arbitrarily fixed constant. The case of constant unknown functions is examined in section 4 and we prove Lipschitz stability in the determination of $\alpha$ from the knowledge of $\Lambda_\alpha$ in this case. Finally, for the sake of completness, we gather several properties (which are not directly useful for the analysis of the inverse problem under study) of the DN map $\Lambda_\alpha$ in Appendix.

\medskip
\noindent
{\bf Acknowledgments.}
We are strongly indebted to the unknown referee of this paper for numerous remarks which have been helpful to us in improving this text.
We also want to thank G\"unther Uhlmann for valuable suggestions and comments during the completion of this work. 


\section{The DN map $\Lambda_\alpha$}
 
\noindent {\bf Solution to the BVP \eqref{2}.}
As we are dealing with an infinitely extended domain $\Omega$, we start by defining the Sobolev spaces on $\partial \Omega$ required by our analysis. Let $s$ be either $1/2$ or $3/2$. Since the trace operator
\begin{align*}
\tau : C_0^\infty (\mathbb{R}, H^{s+1/2}(\omega ))&\longrightarrow L^2(\mathbb{R}, H^s(\partial \omega ))
\\
G&\mapsto \left[t\in \mathbb{R}\mapsto G(t,\cdot )_{|\partial \omega}\right]
\end{align*}
extends to a bounded operator, still denoted by $\tau$, from $L^2(\mathbb{R},H^{s+1/2}(\omega ))$ into $L^2(\mathbb{R}, H^s(\partial \omega ))$, we put
\begin{equation} \label{hs}
\widetilde{H}^{s}(\partial \Omega )=\left\{g\in L^2(\mathbb{R},H^{s}(\partial \omega ));\; \mbox{there exists}\; G\in H^{s+1/2}(\Omega )\; \mbox{such that}\; \tau G=g\right\}.
\end{equation}
The space $\widetilde{H}^{s}(\partial \Omega )$ is Hilbertian (e.g. \cite[p. 398]{Sc}) for the quotient norm 
\begin{equation} \label{hsn}
\|g\|_{\widetilde{H}^{s}(\partial \Omega )}=\min \left\{\|G\|_{H^{s+1/2}(\Omega )};\ G \in H^{s+1/2}(\Omega )\ \mbox{is such that}\ \tau G=g\right\}.
\end{equation}
For the sake of simplicity we write $G=g$ on $\partial \Omega$ instead of $\tau G=g$ in the sequel.

\begin{remark}
The definition in \cite[Chap.~1]{LM} of fractional Sobolev spaces $H^s(\partial \Omega )$, $s \in \mathbb{Q}$, on compact manifolds without boundary, may as well be adapted to the manifold $\partial \Omega = \partial \omega \times \mathbb{R}$ and it is not hard to check that the two spaces $H^s(\partial \Omega )$ and $\widetilde{H}^{s}(\partial \Omega )$ then coincide both algebraically and topologically in this case. 
\end{remark}

\begin{remark}
\label{rmk-min}
There is a minimizer to \eqref{hsn} for any arbitrary $g \in \widetilde{H}^{s}(\partial \Omega )$ with $s \in \{ 1 /2 , 3/2 \}$. Indeed, the mapping $\tau$ is continuous from $H^{s + 1 \slash 2}(\Omega)$ into $L^2(\mathbb{R},H^{s + 1 \slash 2}(\omega))$ so its kernel is closed in $H^{s+1 \slash 2}(\Omega)$. Then $\{ G \in H^{s+ 1 \slash 2}(\Omega),\ \tau G = g \}$ is a closed affine subset of $H^{s+1 \slash 2}(\Omega)$. This yields the existence of a minimizer to \eqref{hsn} since there is always an element of minimal norm in a closed convex subset of an Hilbert space.
\end{remark}

As a direct consequence of (\ref{hs})-(\ref{hsn}) and Remark \ref{rmk-min} we have the following extension lemma:
 
\begin{lemma}
\label{lemma1bis} 
Let $g\in \widetilde{H}^{s}(\partial \Omega )$ for $s \in \{ 1 /2 , 3/2 \}$. Then there exists $G\in H^{s+1/2}(\Omega )$ such that $G=g$ on $\partial \Omega$ and
 $$
 \|G\|_{H^{s+1/2}(\Omega )} = \|g\|_{\widetilde{H}^{s}(\partial \Omega )}.
 $$
 \end{lemma}

The main ingredient in the analysis of the BVP \eqref{2} is the uniform ellipticity of $A$, where $A$ denotes either $A(x',t)$ or $A_0(x',\alpha (x_3))$, as defined in section 1. Indeed, for all $\zeta \in \mathbb{R}^3$ we have
\begin{align*}
A(x',t)\zeta \cdot \zeta  &=\zeta _1^2+\zeta _2^2+\zeta_3^2-2tx_2\zeta _1\zeta _3+2tx_1\zeta _2\zeta _3+t^2(x_2\zeta_1 -x_1\zeta _2)^2
\\
&=\zeta _1^2+\zeta _2^2+(\zeta_3 +t(x_2\zeta_1 -x_1\zeta _2))^2,\ x'=(x_1,x_2) \in \omega,\, t \in {\mathbb R},
\end{align*}
by straightforward computations, which entails that $A(x',t)\zeta \cdot \zeta=0$ if and only if $\zeta =0$. Since $\overline{\omega} \times [\underline{t},\overline{t}]$ is compact for 
all real numbers $\underline{t}<\overline{t}$, there is thus $\lambda \geq 1$, depending on $\omega$, $\underline{t}$ and $\overline{t}$, such that we have
\begin{equation}\label{8}
\lambda ^{-1}|\zeta |^2\leq A(x',t)\zeta \cdot \zeta \leq\lambda |\zeta |^2\; \textrm{for all}\; x'\in \omega ,\; t\in [\underline{t},\overline{t}],\; \zeta \in \mathbb{R}^3.
\end{equation}
We turn now to studying the direct problem associated with \eqref{2}. We pick $f\in \widetilde{H}^{1/2}(\partial \Omega )$ and $F\in H^1(\Omega)$ such that $F=f$ on $\partial \Omega$.  In light of \eqref{8} and the Lax-Milgram lemma, there exists a unique $v\in H^1_0(\Omega)$ solving the variational problem
\begin{equation}\label{25}
\int_\Omega A\nabla v\cdot \nabla w dx =-\int_\Omega A\nabla F\cdot \nabla w dx,\ \textrm{for all}\ w\in H_0^1(\Omega ).
\end{equation}
Hence $u=v+F$ is the unique weak solution to the BVP \eqref{2}. That is, $u$ satisfies the first equation in \eqref{2} in the distributional sense and the second equation in the trace sense. By taking $w=v$ in \eqref{25}, we get $\|v\|_{H^1(\Omega )}\leq C \|F\|_{H^1(\Omega )}$ from \eqref{8} and Poincar\'e's inequality\footnote{Which holds true for $\Omega$ since $\omega$ is bounded.}{,} whence
  \[
 \|u\|_{H^1(\Omega )}\leq C\|F\|_{H^1(\Omega )},
 \]
where $C$ denotes some generic positive constant depending on $\omega$. Finally, by choosing $F \in H^1(\Omega)$ in accordance with Lemma \ref{lemma1bis} so that 
$\| F \|_{H^1(\Omega)}= \| f \|_{\tilde{H}^{1 \slash 2}(\partial \Omega)}$, we find out that
 \begin{equation}\label{26}
 \|u\|_{H^1(\Omega )}\leq C\|f\|_{\widetilde{H}^{1/2}(\partial \Omega )}.
 \end{equation}
 
 
\noindent {\bf Definition of the DN map.}
Let us first introduce the following $H(\textrm{div})$-type space,
 \[
 H(\mdiv_A,\Omega )=\{ P\in L^2(\Omega )^3;\; \mdiv (AP)\in L^2(\Omega )\},
 \]
and prove that $P\in C_0^\infty (\overline{\Omega})\mapsto AP \cdot \nu \in C^\infty (\partial \Omega )$ may be extended to a bounded operator from $H(\mdiv_A,\Omega )$ into the space $\widetilde{H}^{-1/2}(\partial \Omega )$ dual of $\widetilde{H}^{1/2}(\partial \Omega )$.

 \begin{proposition}\label{proposition1}
 Let $P\in H(\mdiv_A,\Omega )$. Then $AP\cdot \nu \in \widetilde{H}^{-1/2}(\partial \Omega )$ and
 we have
\begin{equation}\label{24}
 \|AP\cdot \nu \|_{\widetilde{H}^{-1/2}(\partial \Omega )}\leq C\big( \|P\|_{L^2(\Omega )}+\|\mdiv (AP)\|_{L^2(\Omega )}\big),
 \end{equation}
 for some constant $C=C(\omega,\theta)>0$.
In addition, the identity
\begin{equation}
\label{27}
\langle AP\cdot \nu ,g \rangle = \int_\Omega G \mdiv (AP) dx+\int_\Omega A\nabla G\cdot P dx,
\end{equation}
holds for $g\in \widetilde{H}^{1/2}(\partial \Omega )$ and $G\in H^1(\Omega)$ such that $G=g$ on $\partial \Omega$. Here $\langle \cdot ,\cdot \rangle$ denotes the duality pairing between $\widetilde{H}^{1/2}(\partial \Omega )$ and its dual $\widetilde{H}^{-1/2}(\partial \Omega )$.
 \end{proposition}
 
 \begin{proof}
We first consider the case of $P\in C_0^\infty (\overline{\Omega})^3$. Fix $g\in \widetilde{H}^{1/2}(\partial \Omega )$ and choose $G \in H^1(\Omega)$ such that $G=g$ on $\partial \Omega$. Since $P$ has a compact support, we have
 \begin{equation}\label{28}
 \int_\Omega G\mdiv (AP)dx=-\int_\Omega A\nabla G\cdot Pdx+\int_{\partial \Omega}gAP\cdot \nu d \sigma,
\end{equation}
by Green's formula, whence
 \[
 \Big|\int_{\partial \Omega}gAP\cdot \nu d\sigma \Big|\leq C\|G\|_{H^1(\Omega )}\big( \|P\|_{L^2(\Omega )}+\|\mdiv (AP)\|_{L^2(\Omega )}\big).
 \]
Taking the infimum over $\{ G\in H^1(\Omega),\ G=g\ \mathrm{on}\ \partial \Omega \}$ in the right hand side of above estimate, we find that \eqref{24} holds true for every $P \in C_0^\infty (\overline{\Omega})^3$.
 
Further, we pick $P \in H(\mdiv_A,\Omega )$. The set $C_0^\infty (\overline{\Omega})^3$ being dense in $H(\mdiv_A,\Omega )$, as can be seen by mimicking the proof of \cite[Theorem~2.4]{GR}, we may find a sequence $(P_k)_k$ in $C_0^\infty (\overline{\Omega})^3$ converging to $P$ in $H(\mdiv_A,\Omega )$. Moreover, $(AP_k\cdot \nu )_k$ being a Cauchy sequence in $\widetilde{H}^{-1/2}(\partial \Omega )$ by  \eqref{24}, then $(AP_k\cdot \nu )_k$ admits a limit, denoted by $AP\cdot \nu$, in $\widetilde{H}^{-1/2}(\partial \Omega )$. Now
\eqref{27} follows readily from this and \eqref{28}.
 \end{proof}
 
Let $u$ be the $H^1(\Omega)$-solution to \eqref{2}. Applying Proposition \ref{proposition1} to $P=\nabla u$, we deduce from \eqref{26} that 
 \[
 \Lambda _\alpha :f \mapsto A\nabla u\cdot \nu
 \]
 is well defined as a bounded operator from $\widetilde{H}^{1/2}(\partial \Omega )$ into $\widetilde{H}^{-1/2}(\partial \Omega )$. Moreover the following identity
 \begin{equation}\label{29}
 \langle \Lambda _\alpha f,g\rangle =\int_\Omega A\nabla u\cdot \nabla G dx,
 \end{equation}
holds for all $g\in \widetilde{H}^{1/2}(\partial \Omega )$ and $G\in H^1(\Omega)$ such that $G=g$ on $\partial \Omega$.
 
Further, by taking $G=v$ in \eqref{29}, where $v$ is the solution to \eqref{2} with $f$ replaced by $g$, we find out that
 \[
 \langle \Lambda _\alpha f,g\rangle =\int_\Omega A\nabla u\cdot \nabla v dx=\int_\Omega \nabla u\cdot A\nabla v dx,
 \]
entailing
 \begin{equation}\label{30}
 \langle \Lambda _\alpha f,g\rangle =\langle  f, \Lambda _\alpha g\rangle ,\; \textrm{for all}\; f,g\in \widetilde{H}^{1/2}(\partial \Omega ).
 \end{equation}
This proves that $\Lambda _\alpha ^\ast{ _{\big|\widetilde{H}^{1/2}(\partial \Omega )}}=\Lambda _\alpha$, where $\widetilde{H}^{1/2}(\partial \Omega )$ is identified with a subspace of its bidual space.

Finally, for $i=1,2$, we put $A_i=A(x',\alpha _i(x_3))$ and $\Lambda _i=\Lambda_{\alpha _i}$, and let $u_i\in H^1(\Omega )$ denote a weak solution to the equation
\[
\mdiv (A_i\nabla u_i)=0\; \textrm{in}\; \Omega .
\]
Applying \eqref{29} to $f=u_i{_{\big|\partial \Omega}}$ and $g=u_{3-i}{_{\big|\partial \Omega}}$, we get
\[ \langle \Lambda _1 u_1,u_2\rangle =\int_\Omega A_1\nabla u_1\cdot \nabla u_2 dx\ \mathrm{and}\
\langle \Lambda _2 u_2,u_1\rangle =\int_\Omega A_2\nabla u_2\cdot \nabla u_1 dx, \]
which yields
\begin{equation}\label{31}
\langle (\Lambda _1 -\Lambda _2)u_1,u_2\rangle =\int_\Omega (A_1-A_2)\nabla u_1\cdot \nabla u_2 dx,
\end{equation}
by \eqref{30}.


\section{Can we determine the unknown function $\alpha$ from $\Lambda_\alpha$?}

In this section we examine the inverse problem of identifying $\alpha$ from the knowledge of $\Lambda_{\alpha}$.

\smallskip
\noindent {\bf Analysis of the problem for general unknown functions.} Let $\gamma$ be a nonempty open subset of $\partial \omega$. We put $\Gamma =\gamma \times (-2L,2L)$ for some fixed $L>0$, and define the functional space
\begin{equation}
\label{eq1a}
\widetilde{H}^{1/2}_\Gamma (\partial \Omega )=\{f\in \widetilde{H}^{1/2}(\partial \Omega );\; \textrm{supp} f\subset \Gamma \}.
\end{equation}

For $\alpha_i\in W^{1,\infty}(\mathbb{R})$, $i=1,2$, put $M=\max_{i=1,2} \| \alpha_i \|_{W^{1,\infty}(\mathbb{R})}$. We note $A_i=A(x',\alpha_i(x_3))$ and $\Lambda_i=\Lambda_{\alpha _i}$. In light of \eqref{31}, we have
\begin{equation}\label{32}
\langle (\Lambda _1 -\Lambda _2)u_1,u_2\rangle =\int_\Omega (A_1-A_2)\nabla u_1\cdot \nabla u_2 dx,
\end{equation}
for any function $u_i \in H^1(\Omega)$ which is a weak solution to the equation
$$\mdiv (A_i\nabla u_i)=0\ \mbox{in}\ \Omega,\ i=1,2. $$ 
Putting $\Omega^L=\omega \times (-L,L)$ and assuming that
\begin{equation}
\label{eq1b}
\alpha _1(x_3)=\alpha _2(x_3),\ |x_3|>L,
\end{equation}
we may rewrite \eqref{32} as
\begin{equation}\label{33}
\langle (\Lambda _1 -\Lambda _2)u_1,u_2\rangle =\int_{\Omega^L}(A_1-A_2)\nabla u_1\cdot \nabla u_2 dx.
\end{equation}

The following analysis is essentially based on the method built in \cite{AG2} for bounded domains, that will be adapted to the case of the infinite waveguide $\Omega$ under consideration. To this purpose we set
$$ \Gamma _\rho =\{ x\in \Gamma ;\; \textrm{dist}(x,\partial \Gamma )>\rho \}\ \textrm{and}\
U_\rho=\{ x\in \mathbb{R}^3 ;\; \textrm{dist}(x,\Gamma _\rho )<\rho /4 \}, $$
for all $\rho \in (0, \rho _0]$, where $\rho _0$ is some characteristic constant depending
only on $\omega$ and $L$, which is defined in \cite{AG2}.
Upon eventually shortening $\rho_0$, we can assume without loss of generality that $\Gamma _0=\gamma _0\times [-L,L]\subset \Gamma _\rho$ for some $\gamma _0 \Subset \gamma$. In view of \cite{AG2}, we thus may find a Lipschitz domain $\Omega _\rho$ satisfying simultaneously:
\[
\Omega \subset \Omega _\rho,\; \Gamma _0\subset\partial \Omega  \cap \Omega _\rho \Subset \Gamma\;  \mathrm{and}\
\textrm{dist}(x,\partial \Omega _\rho)\geq \rho /2\; \textrm{for all}\; x\in U_\rho.
\]
Moreover we know from \cite[Section~3]{AG1} that there is a unitary $C^\infty$ vector field $\tilde{\nu}$, defined in some suitable neighborhood of $\partial \omega\times (-2L,2L)$, which is non tangential to $\partial \Omega$ and points to the exterior of $\Omega$.
For $x^0\in \overline{\Gamma _\rho}$ we pick $\tau_0>0$ sufficiently small so the point $z_\tau = x^0+\tau \tilde{\nu}(x^0)$ obeys $C \tau \leq \textrm{dist}(z_{\tau},\partial \Omega) \leq \tau$ for every
$\tau \in(0, \tau_0]$, by \cite[Lemma~3.3]{AG1}. 
Here $\tau_0$ and $C$ are two positive constants depending only on $\Omega$, and $\lambda$, $\underline{t}$ and $\overline{t}$ are the same as in \eqref{8}.

\smallskip
In light of \cite{HK}, the operator $\mdiv (A_i\nabla \cdot\, )$, $i=1,2$, has a Dirichlet Green function $G_i=G_i(x,y)$ on $\Omega _\rho$.
More specifically, $G_i(x)=G_i(x,z_\tau )$ is the solution to the BVP
\begin{equation}
\label{equation1}
\left\{
\begin{array}{ll}
\mdiv (A_i\nabla G_i)=-\delta (x-z_\tau )\;\; &\textrm{in}\;\; \mathscr{D}'(\Omega _\rho)
\\
G_i=0 &\textrm{on}\;\; \partial \Omega _\rho,
\end{array}
\right.\ i=1,2,
\end{equation}
and we will see in the coming lemma that the claim of \cite[Corollary~3.4]{AG2} remains essentially unchanged for the unbounded domain $\Omega _\rho$ arising in this framework.

\begin{lemma}\label{lemma1}
There are two constants $\tau _0=\tau _0(\Omega ,\rho) >0$ and $C=C(\Omega ,\lambda ,\underline{t},\overline{t})$ such that the restriction $G_i(.,z_\tau)$ to $\Omega$, $i=1,2$, belongs to $H^1(\Omega )$, and satisfies
\begin{equation}\label{equation2}
\|G_i(\cdot ,z_\tau )\|_{H^1(\Omega )}\leq C\tau ^{-1/2},\;\; 0<\tau \leq \tau _0.
\end{equation}
\end{lemma}
\begin{proof}
Let the space $Y^{2,1}(\Omega )=\{u\in L^6(\Omega );\; \nabla u\in L^2(\Omega )^3\}$
be endowed with the norm $\| v \|_{Y^{2,1}(\Omega )}=\| v \|_{L^6(\Omega)} + \| \nabla v \|_{L^2(\Omega)^3}$ and
choose $\tau_0=\tau_0(\Omega,\rho)$ so small relative to $\rho$ that 
$\textrm{dist}(z_\tau,\partial \Omega_\rho)>\tau$ for all $\tau \in (0,\tau_0]$.
Since the ball $B_{\tau \slash 3}(z_\tau)$ centered at $z_\tau$ with radius $\tau \slash 3$, is embedded in $\Omega_\rho \setminus \overline{\Omega}$, we have
\begin{equation}
\label{equ3}
\|G_i(\cdot ,z_\tau )\|_{Y^{2,1}(\Omega )}\leq \| G_i(\cdot,z_{\tau}) \|_{Y^{2,1}(\Omega_\rho \setminus B_{\tau \slash 3}(z_\tau))} \leq C\tau ^{-1/2},\;\; 0<\tau \leq \tau _0,
\end{equation}
directly from \cite[Formula~(4.45)]{HK}. Here $C$ is some positive constant depending only on $\Omega$, $\lambda$, $\underline{t}$ and $\overline{t}$.

Finally, since the Poincar\'e inequality holds in $\Omega$ because $\omega$ is bounded, we may deduce (\ref{equation2}) from (\ref{equ3}).
\end{proof}

Further, as a Green function is a Levi function\footnote{Let $E=\sum_{1 \leq i,j \leq m} \alpha_{ij}(x) \partial ^2_{ij}$ be an elliptic operator acting in a subdomain of $\mathbb{R}^m$, $m\geq 3$, and let $\beta (x)=(\beta_{ij}(x))_{1 \leq i,j \leq m}$ be the inverse matrix to $\alpha (x)=(\alpha_{ij}(x))_{1 \leq i,j \leq m}$. Then $H(x,y)=|\mathbb{S}^{m-1}|^{-1} |\mbox{det}(\alpha (y))|^{-1/2} \left[\beta (y)(x-y)\cdot (x-y)\right]^{\frac{2-m}{2}}$ is solution to the equation $\sum_{1 \leq i, j \leq m} \alpha_{ij}(y) \partial_{ij}^2 H(\cdot,y)=0$, and is called a parametrix of $E$. With reference to \cite{M}, any function $L=L(x,y)$ which is smooth outside $x\neq y$ and satisfies $\partial ^\alpha (L-H)(x,y)=O(|x-y|^{2+\lambda -|\alpha |-m})$ for some $\lambda>0$ and any $\alpha \in \mathbb{N}^m$, $|\alpha |\leq 2$, is a Levi function.}, it behaves locally like a parametrix. 
Hence, by applying \cite[Formula~(8.4)]{M}, $G_i$ can be brought into the form
\begin{equation}
\label{e1}
G_i(x)=C(\textrm{det}\, (A_i(z_\tau )))^{-1/2}\big(A_i(z_\tau )^{-1}(x-z_\tau )\cdot (x-z_\tau )\big)^{-1/2}+R_i(x),\ i=1,2,
\end{equation}
where $A_i(z_\tau)$ is a shorthand for $A(z_\tau',\alpha_i((z_\tau)_3))$, $C>0$ is a constant, and the reminder $R_i$ obeys the conditions
\[
\exists (r_0,\alpha) \in \mathbb{R}_+^* \times (0,1),\ |R_i(x)|+|x-z_\tau|\, |\nabla R_i|\leq C|x-z_\tau|^{-1+\alpha},\
x \in \Omega _\rho, |x-z_\tau|\leq r_0.
\]
On the other hand it holds true for $i=1,2$, that $\mdiv (A_i\nabla G_i)=0$ in the weak sense in $\Omega$, so \eqref{33} entails
\[
\langle (\Lambda _1 -\Lambda _2)G_1,G_2\rangle =\int_{\Omega^L}(A_1-A_2)\nabla G_1\cdot \nabla G_2dx.
\]
Taking into account that $G_i{_{\big |\partial \Omega}}\in \widetilde{H}^{1/2}(\partial \Omega )$, this yields
\[
\left| \int_{\Omega^L}(A_1-A_2)\nabla G_1\cdot \nabla G_2dx \right| \leq \|\Lambda _1 ^\Gamma -\Lambda _2 ^\Gamma \| \| G_1\|_{\widetilde{H}^{1/2}(\partial \Omega )}  \|G_2 \|_{\widetilde{H}^{1/2}(\partial \Omega )},
\]
where $\Lambda_i ^\Gamma$ is the restriction of $\Lambda_i$ to the closed subspace $\widetilde{H}^{1/2}_\Gamma (\partial \Omega )$, and consequently
\begin{equation}\label{e2}
\left|  \int_{\Omega^L}(A_1-A_2)\nabla G_1\cdot \nabla G_2dx \right| \leq C \|\Lambda _1 ^\Gamma -\Lambda _2 ^\Gamma \| \| G_1\|_{H^1( \Omega )}  \|G _2\|_{H^1(\Omega )}.
\end{equation}

The next step involves picking $x^0 \in \Gamma_0$ such that $|\alpha_1(x_3^0)-\alpha_2(x_3^0)|=\|\alpha_1-\alpha _2\|_{L^\infty (-L,L)}.$
We may actually assume without loss of generality that $|\alpha_1(x_3^0)-\alpha_2(x_3^0)|=\alpha_1(x_3^0)-\alpha_2(x_3^0)$, hence
\begin{equation}
\label{xze}
\alpha_1(x_3^0)-\alpha_2(x_3^0)=\|\alpha_1-\alpha _2\|_{L^\infty (-L,L)}.
\end{equation}
Let us next decompose the left hand side of \eqref{e2} into the sum
\begin{equation}\label{eq4}
 \int_{\Omega ^L}(A_1-A_2)\nabla G_1\cdot \nabla G_2dx= \int_{\Omega^L\cap B(z_\tau ,\rho  )}(A_1-A_2)\nabla G_1\cdot \nabla G_2dx+ \int_{\Omega^L\setminus B(z_\tau ,\rho  )}(A_1-A_2)\nabla G_1\cdot \nabla G_2dx,
\end{equation}
where, as previously stated, $z_{\tau}=x^0+\tau \tilde{\nu}(x^0)$ for some $\tau \in \mathbb{R}_+^*$.
Recalling \eqref{e1} we see that the first term in the right hand side of \eqref{eq4} is lower bounded by
$$\int_{B(z_\tau ,\rho) \cap \Omega ^L} \frac{A_2(z_\tau)^{-1} \left( A_1(x)-A_2(x) \right) A_1(z_\tau)^{-1} (x-z_\tau ) \cdot (x-z_\tau )}{\left( A_1(z_\tau )^{-1} \cdot (x-z_\tau ) \right)^{3/2} \left( A_2(z_\tau)^{-1} \cdot (x-z_\tau ) \right)^{3/2}}dx
-c \int_{\Omega^L\cap B(z_\tau ,\rho)}|x-z_\tau |^{-4+2\alpha}dx, $$
where $c$ is a positive constant depending only on $M$ and $\omega$.
From this and the Lipschitz continuity of $x \mapsto A_i(x)$, $i=1,2$, then follows upon writing $A_i(x^0)$ instead of $A((x^0)',\alpha_i(x_3^0))$, that
\begin{eqnarray}
& & \int_{\Omega^L \cap B(z_\tau ,\rho)} (A_1-A_2) \nabla G_1 \cdot \nabla G_2 dx \nonumber \\
& \geq & \int_{B(z_\tau ,\rho)\cap \Omega ^L} \frac{\left( A_1(x^0)^{-1})-A_2(x^0)^{-1} \right)(x-z_\tau )\cdot (x-z_\tau )}{\left( A_1(z_\tau)^{-1} (x-z_\tau) \cdot (x-z_\tau) \right)^{3/2} \left( A_2(z_\tau)^{-1} (x-z_\tau) \cdot (x-z_\tau) \right)^{3/2}}dx \nonumber \\
& - & c \int_{\Omega^L\cap B(z_\tau ,\rho)} |x-z_\tau|^{-4} \left( |x-x^0|^{2 \alpha} + |x^0 - z_\tau |^{2 \alpha} \right)dx. \label{eq1}
\end{eqnarray}
On the other hand, since $\left| \int_{\Omega^L\setminus B(z_\tau ,\rho)}(A_1-A_2)\nabla G_1\cdot \nabla G_2dx \right|$ is majorized by, say, $c$, and
$$
\|\Lambda _1 ^\Gamma -\Lambda _2 ^\Gamma \| \| G_1\|_{H^1( \Omega )}  \|G _2\|_{H^1(\Omega )}\leq C \tau ^{-1}\|\Lambda _1 ^\Gamma -\Lambda _2 ^\Gamma \|,
$$
by Lemma \ref{lemma1}, we deduce from \eqref{eq4}-\eqref{eq1} that
\begin{eqnarray}
& & \int_{B(z_\tau ,\rho)\cap \Omega ^L} \frac{\left( A_1(x^0)^{-1}-A_2(x^0)^{-1} \right)(x-z_\tau )\cdot (x-z_\tau )}{\left( A_1(z_\tau)^{-1} (x-z_\tau) \cdot (x-z_\tau) \right)^{3/2} \left( A_2(z_\tau)^{-1} (x-z_\tau) \cdot (x-z_\tau) \right)^{3/2}}dx \nonumber \\
& \leq & C \left( \tau ^{-1+\alpha} +1+\tau ^{-1}\|\Lambda _1 ^\Gamma -\Lambda _2 ^\Gamma \|\right). \label{eq2}
\end{eqnarray}

The main ingredient in the analysis developed in \cite{AG2} is the ellipticity condition \cite[Formula~(2.5)]{AG2} imposed on $\partial _tA(x',t)$,
implying
\begin{align}
&\int_{B(z_\tau ,\rho)\cap \Omega ^L} \frac{\left( A_1(x^0)^{-1}-A_2(x^0)^{-1} \right)(x-z_\tau )\cdot (x-z_\tau )}{\left( A_1(z_\tau)^{-1} (x-z_\tau) \cdot (x-z_\tau) \right)^{3/2} \left( A_2(z_\tau)^{-1} (x-z_\tau) \cdot (x-z_\tau) \right)^{3/2}}dx \label{e3}
\\
&\hskip 10cm \geq C_0 \tau^{-1} (\alpha_1(x_3^0)-\alpha_2(x_3^0) ),\nonumber
\end{align}
for some constant $C_0>0$. This, \eqref{xze} and \eqref{eq2} finally lead to the desired stability estimate:
$$
\|\alpha_1-\alpha _2\|_{L^\infty (-L,L)} \leq C\|\Lambda _1 ^\Gamma -\Lambda _2 ^\Gamma \|.
$$

Unfortunately, it turns out that \cite[Formula~(2.5)]{AG2} is not fulfilled by $\partial_t A(x',t)$ in this framework. This can be seen from the following explicit expression
\begin{equation}
\label{ev}
\lambda_1(x',t)=0,\ \lambda_2(x',t)= |x'|^2t-\sqrt{|x'|^4t^2+|x'|^2},\ \lambda_3(x',t)=|x'|^2t+\sqrt{|x'|^4t^2+|x'|^2},
\end{equation}
of the eigenvalues of $\partial_t A(x',t)$, showing that the spectrum of $\partial_t A(x',t)$ has a negative component
for $x' \in \partial \omega$. Moreover, due to the occurence of this negative eigenvalue, the weak monotonicity assumption \cite[Formula~(5.7)]{AG1} is not satisfied by the conductivity matrix under consideration either. Therefore, the approach developed in \cite{AG2} does not apply to the inverse problem of determining $\alpha$ from the knowledge of $\Lambda_\alpha$, which remains open in the general case. 

\smallskip
Nevertheless, we will see in section \ref{sec-constant} that this is not the case for constant unknown functions anymore. But, prior to examining this peculiar framework, we will now deduce from the above reasoning, upon substituting a suitable matrix $A^\bullet$ for $A$, that unknown functions $\alpha$ which are close to some {\it a priori} fixed constant value may well be identified from the associated DN map.\\

\noindent {\bf The case of unknown functions close to a constant value.} 
Put
\[
A^\bullet (x',t)=t
 \begin{pmatrix}
1+x_2^2 & -x_2x_1 &-x_2
\\
-x_2x_1 & 1+x_1^2 & x_1
\\
-x_2& x_1 &1
 \end{pmatrix},\ x'\in \omega ,\ t\in \mathbb{R},
 \]
so we have $A^\bullet (x',t)=tA^\bullet (x',1)$.

We denote by $\Lambda _\alpha ^\bullet$ the DN map $\Lambda_{\alpha}$ where $A^\bullet (x',\alpha (x_3))$ is substituted for
$A(x',\alpha (x_3))$. Then, by arguing as in the derivation of \eqref{31}, we obtain that
 \begin{equation}\label{51}
\langle (\Lambda _\alpha -\Lambda _\alpha ^\bullet )u,u^\bullet \rangle =\int_\Omega (A(x',\alpha (x_3) )-A^\bullet (x',\alpha (x_3))\nabla u\cdot \nabla u^\bullet dx,
\end{equation}
for all $u, u^{\bullet} \in H^1(\Omega )$ satisfying
$\mdiv (A(x',\alpha (x_3) )\nabla u)=\mdiv (A^\bullet (x',\alpha (x_3) )\nabla u^\bullet )=0$
 in the weak sense in $\Omega$. By direct calculation we notice from the definitions of $A$ and $A^\bullet$ that 
$$
A(x,t)-A^\bullet (x',t)=(1-t) \left( B_0 - t  B_1(x') \right),\ (x',t)\in \omega \times \mathbb{R}, 
$$
where $B_0$ stands for the identity matrix of $\mathbb{R}^3$ and 
$$
B_1(x')= \begin{pmatrix}
x_2^2 & -x_2x_1&0 \\
-x_2x_1 & x_1^2 & 0\\
0& 0 &0
\end{pmatrix},\ x' \in \omega.
$$
Assume that $\alpha \in L^{\infty}(\mathbb{R})$,  $\| \alpha  -1\|_{L^{\infty}(\mathbb{R})}\leq M$.
Given $f$ (resp., $g$) in $\tilde{H}^{1 \slash 2}(\partial \Omega)$, we call $u$ (resp., $u^\bullet$) the solution to \eqref{2} (resp., \eqref{2} where $(A^\bullet,g)$ is substituted to $(A,f)$). Applying \eqref{51}, we get that
$| \langle (\Lambda _\alpha -\Lambda _\alpha^\bullet )f, g \rangle | \leq c(M,\omega) \| \alpha - 1 \|_{\infty} \| u \|_{H^1(\Omega)} \| u^\bullet \|_{H^1(\Omega)}$, where $C=C(M,\omega)$ is some positive constant depending only on $M$ and $\omega$. Therefore we have
$| \langle (\Lambda _\alpha -\Lambda _\alpha ^\bullet )f, g \rangle | \leq C \| \alpha - 1 \|_{\infty}
\| f \|_{\tilde{H}^{1 \slash 2}(\partial \Omega)} \| g \|_{\tilde{H}^{1 \slash 2}(\partial \Omega)}$, by \eqref{26}, hence
\begin{equation}
\label{eq-approx}
\|\Lambda _\alpha -\Lambda ^\bullet _\alpha \|_{\mathscr{L}(\widetilde{H}^{1/2}(\partial \Omega ),\widetilde{H}^{-1/2}(\partial \Omega )) }\leq C\|\alpha -1\|_\infty,
\end{equation}
showing that $\Lambda_{\alpha}^{\bullet}$ is a suitable approximation of $\Lambda_{\alpha}$ provided $\alpha$ is sufficiently close to $1$.\footnote{Note that $1$ can be replaced by any constant $\mu \neq 0$ upon substituting $A^\bullet (\mu x',\frac{t}{\mu})$ for $A^\bullet (x',t)$ in the above reasoning, since $A(x,t)=A(\mu x',\frac{t}{\mu})$.}

Actually, the main benefit of dealing with $\Lambda_{\alpha}^{\bullet}$ instead of $\Lambda_{\alpha}$  in the inverse problem of determining $\alpha$ from $\Lambda_{\alpha}$, boils down to the following identity
$$ \partial _tA^\bullet (x',t)=A^\bullet (x',1)=A(x',1),\ {\rm for\ all}\ x' \in \omega\ {\rm and}\ t \in {\mathbb R}, $$
ensuring that the ellipticity condition \cite[Formula~(2.5)]{AG2} required by the method presented in \cite{AG2}, is verified by $A^{\bullet}$. 
Now, in light of the reasoning developed in the first part of this section, which is borrowed from the proof of \cite[Theorem~2.2]{AG2}, we derive the:

\begin{theorem}
\label{thm-aconstant}
For $L>0$, $\underline{t}>0$ and $M>0$ fixed, let $\alpha_i \in W^{1,\infty}(\mathbb{R})$, $i=1,2$, obey \eqref{eq1b} and fulfill
$\alpha_i \geq \underline{t}$ and $\|\alpha_i\|_{W^{1,\infty }(\mathbb{R})} \leq M$.
Then there exists a constant $C>0$ depending only on $\omega$, $M$ and $L$, such that we have
\[
\| \alpha_1-\alpha_2\|_{L^\infty (\mathbb{R})}\leq C
\|(\Lambda ^\bullet _{\alpha _1})^\Gamma -(\Lambda ^\bullet _{\alpha _2})^\Gamma \|_{\mathscr{L}(\widetilde{H}^{1/2}(\partial \Omega ),\widetilde{H}^{-1/2}(\partial \Omega )) },
\]
where $(\Lambda ^\bullet _{\alpha _i})^\Gamma=(\Lambda ^\bullet _{\alpha _i})_{| \widetilde{H}^{1/2}_\Gamma (\partial \Omega )}$ for $i=1,2$, the space $\widetilde{H}^{1/2}_\Gamma (\partial \Omega )$ being the same as in \eqref{eq1a}.
\end{theorem}

 
\section{The case of constant unknown functions}
\label{sec-constant}

In this section we address the case of affine twisting functions $\theta$, that is constant functions $\alpha$, by means of the partial Fourier transform
$\mathcal{F}_{x_3}$ with respect to the variable $x_3$. This is suggested by the translational invariance of the system under consideration in the infinite direction $x_3$, arising from the fact that the matrix $A(x',\alpha (x_3))$ appearing in \eqref{2} does not depend on $x_3$ in this peculiar case. 

\smallskip
In the sequel, we note $\xi$ the Fourier variable associated with $x_3$ and we write $\widehat{w}$ instead of ${\mathcal F}_{x_3} w$ for every function $w=w(x',x_3)$:
$$ \widehat{w}(x',\xi)=({\mathcal F}_{x_3} w)(x',\xi),\ x' \in \omega,\ \xi \in \mathbb{R}. $$

\smallskip
The first step of the method is to re-express the system \eqref{2} in the Fourier plane $\{ (x',\xi),\ x' \in \omega, \xi \in \mathbb{R} \}$.\\

\noindent {\bf Rewriting the BVP in the Fourier variables.}
We start with two useful technical lemmas.
\begin{lemma}
\label{lemma2}
For every $w\in H^1(\Omega)$ the identity $\widehat{\partial _{x_j}w}=\partial _{x_j}\widehat{w}$ holds for $j=1,2$.
\end{lemma}

\begin{proof}
Fix $j=1,2$. For every $\varphi \in C_0^\infty (\omega )$ and $\psi \in \mathcal{S}(\mathbb{R})$, we have
\[
\int_\omega \varphi (x') \left( \int_\mathbb{R}\partial _{x_j}w(x',x_3)\widehat{\psi}(x_3)dx_3 \right) dx'=\int_\mathbb{R}\widehat{\psi}(x_3) \left( \int_\omega \partial _{x_j}w(x',x_3)\varphi (x')dx' \right) dx_3,
\]
from Fubini's theorem. By integrating by parts in the last integral, we obtain
\[
\int_\omega \partial _{x_j}w(x',x_3)\varphi (x')dx'=-\int_\omega w(x',x_3)\partial _{x_j}\varphi (x')dx',\; \textrm{a.e.}\; x_3\in \mathbb{R},
\]
so we get
\begin{eqnarray*}
\int_\omega \varphi (x') \left( \int_\mathbb{R}\partial _{x_j}w(x',x_3)\widehat{\psi}(x_3)dx_3 \right) dx'& = & -\int_\mathbb{R}\widehat{\psi}(x_3) \left( \int_\omega w(x',x_3)\partial _{x_j}\varphi (x')dx'\right) dx_3 \\
& = & 
- \int_\omega \partial _{x_j}\varphi (x') \left( \int_\mathbb{R}w(x',x_3)\widehat{\psi}(x_3)dx_3 \right) dx'.
\end{eqnarray*}
Further, the operator $\mathcal{F}_{x_3}$ being selfadjoint in $L^2(\mathbb{R})$, it is true that
\[
\int_\mathbb{R}w(x',x_3)\widehat{\psi}(x_3)dx_3=\int_\mathbb{R}\widehat{w}(x',\xi ) \psi (\xi )d\xi,\; \textrm{a.e.}\; x'\in \omega,
\]
whence
\begin{align*}
\int_\omega \varphi (x') \left( \int_\mathbb{R}\partial _{x_j}w(x',x_3)\widehat{\psi}(x_3)dx_3 \right) dx'&=-\int_\omega \partial _{x_j}\varphi (x') \left( \int_\mathbb{R}\widehat{w}(x',\xi ) \psi (\xi )d\xi \right) dx'\\ &= \int_\omega \varphi (x') \left( \int_\mathbb{R}\partial _{x_j}\widehat{w}(x',\xi ) \psi (\xi )d\xi \right) dx',
\end{align*}
by integrating by parts. From the density of $C_0^{\infty}(\omega)$ in $L^2(\omega)$, the above identity entails that
\[
\int_\mathbb{R}\partial _{x_j}w(x',x_3)\widehat{\psi}(x_3)dx_3=\int_\mathbb{R}\partial _{x_j}\widehat{w}(x',\xi ) \psi (\xi )d\xi,\; \textrm{a.e.}\; x'\in \omega,
\]
for every $\psi \in \mathcal{S}(\mathbb{R})$. From this, the selfadjointness of $\mathcal{F}_{x_3}$ and the density of $\mathcal{S}(\mathbb{R})$ in $L^2(\mathbb{R})$, then follows that
$\widehat{\partial _{x_j}w}=\partial _{x_j}\widehat{w}$.
\end{proof}

\begin{lemma}\label{lemma3}
Let $C=(C_{kl})_{1 \leq k,l \leq 3} \in W^{1,\infty}(\omega )^{3\times 3}$ be such that $(C_{kl}(x'))_{1 \leq k,l \leq 3}$ is symmetric for all $x'\in \omega$. Then every $w\in H^1(\Omega )$ obeying
\begin{equation}\label{34}
\int_\Omega C\nabla w\cdot \nabla v dx =0\ \textrm{for all}\ v\in H_0^1(\Omega ),
\end{equation}
satisfies the equation
\begin{equation}
\label{eqdiv}
-\mdiv_{x'} (\widetilde{C}(x')\nabla _{x'}\widehat{w})+P(x',\xi )\cdot \nabla _{x'}\widehat{w}+q(x',\xi )\widehat{w}=0\;\; \textrm{in}\; \mathcal{D}'(\Omega ),
\end{equation}
with
\begin{align*}
&\widetilde{C}(x')=(C_{ij}(x'))_{1\leq i,j \leq 2}
\\
&P(x',\xi )=-i2\xi
 \begin{pmatrix}
 C_{31}(x')  \\ C_{32}(x') 
 \end{pmatrix}
 \\
&q(x',\xi )=-i\xi \mdiv_{x'} \begin{pmatrix}
 C_{31}(x')  \\ C_{32}(x') 
 \end{pmatrix} +\xi ^2C_{33}(x'),\ (x',\xi) \in \Omega.
\end{align*}
Moreover, if $w\in H^2(\Omega )$ is solution to \eqref{34} then \eqref{eqdiv} holds for a.e. $(x',\xi) \in \Omega$.
\end{lemma}

\begin{proof}
Choose $v=\varphi \otimes \widehat{\psi}$ in \eqref{34}, with $\varphi \in C_0^\infty (\omega )$ and $\psi \in \mathcal{S}(\mathbb{R})$, so we have:
\begin{equation}\label{35}
\sum_{k,l=1,2,3}\int_\Omega C_{kl}(x')\partial_{x_k}w(x',x_3)\partial _{x_l}(\varphi \otimes \widehat{\psi})(x',x_3)dx'dx_3=0.
\end{equation}
For $1 \leq k, l \leq 2$, we notice that
\begin{align}
\int_\Omega C_{kl}(x')\partial_{x_l}w(x',x_3)\partial _{x_k}(\varphi \otimes \widehat{\psi})(x',x_3)dx'dx_3
& =  \int_\omega C_{kl}(x')\partial _{x_k}\varphi (x') \left( \int_\mathbb{R}\partial_{x_l} w(x',x_3)\widehat{\psi}(x_3) dx_3 \right) dx' \nonumber \\
& =   \int_\Omega C_{kl}(x')\partial_{x_l}\widehat{w}(x',\xi )\partial _{x_k}(\varphi \otimes \psi)(x',\xi )dx'd\xi, \label{36}
\end{align}
directly from Lemma \ref{lemma2}. Further, the identity
\begin{align*}
\int_\Omega C_{3l}(x')\partial_{x_l}w(x',x_3)\partial _{x_3}(\varphi \otimes \widehat{\psi})(x',x_3)dx'dx_3 &=\int_\omega C_{3j}(x')\varphi (x') \left( \int_\mathbb{R}\partial_{x_l}w(x',x_3)\widehat{\psi}'(x_3)dx_3 \right) dx'
\\ 
&=\int_\omega C_{3l}(x')\varphi (x') \left( \int_\mathbb{R}\partial_{x_l}w(x',x_3)\widehat{(-i\xi )\psi}(x_3)dx_3 \right) dx',
\end{align*}
holds for $l=1,2$, so we have
\begin{equation}\label{37}
\int_\Omega C_{3l}(x')\partial_{x_l}w(x',x_3)\partial _{x_3}(\varphi \otimes \widehat{\psi})(x',x_3)dx'dx_3=\int_\Omega C_{3l}(x')(-i\xi )\partial_{x_l}\widehat{w}(x',\xi )(\varphi \otimes \psi)(x',\xi )dx'd\xi .
\end{equation}
Next, since
\begin{align*}
\int_\Omega C_{k3}(x')\partial_{x_3}w(x',x_3)\partial _{x_k}(\varphi \otimes \widehat{\psi})(x',x_3)dx'dx_3
& =  \int_\omega C_{k3}(x')\partial_{x_k}\varphi (x') \left( \int_\mathbb{R}\partial_{x_3}w(x',x_3)\widehat{\psi}(x_3)dx_3 \right) dx' \\
& =  - \int_\omega C_{k3}(x')\partial_{x_k}\varphi (x') \left( \int_\mathbb{R}w(x',x_3)\widehat{\psi}'(x_3)dx_3 \right)
dx',
\end{align*}
for each $k=1,2$, with
\[
\int_\mathbb{R}w(x',x_3)\widehat{\psi}'(x_3)dx_3
=-\int_\mathbb{R}w(x',x_3)\widehat{(i \xi) \psi}(x_3)dx_3 = -\int_\mathbb{R}  (i \xi) \widehat{w}(x',\xi) \psi(\xi)d \xi,
\]
then
\[
\int_\Omega C_{k3}(x')\partial_{x_3}w(x',x_3)\partial _{x_k}(\varphi \otimes \widehat{\psi})(x',x_3)dx'dx_3=\int_\Omega C_{k3}(x')(i\xi )\widehat{w}(x',\xi )\partial_{x_k}(\varphi \otimes \psi)(x',\xi )dx'd\xi .
\]
Integrating by parts in the right hand side of the above identity we get
\begin{align}\label{38}
\int_\Omega C_{k3}(x')\partial_{x_3}w(x',x_3)\partial _{x_k}(\varphi \otimes \widehat{\psi})(x',x_3)dx'dx_3&=\int_\Omega C_{k3}(x')(-i\xi )\partial_{x_k}\widehat{w}(x',\xi )(\varphi \otimes \psi)(x',\xi )dx'd\xi \nonumber
\\
&+\int_\Omega \partial_{x_k}C_{k3}(x')(-i\xi )\widehat{w}(x',\xi )(\varphi \otimes \psi)(x',\xi )dx'd\xi .
\end{align}
Further, bearing in mind that $\partial _{x_3}(\varphi \otimes \widehat{\psi})=\varphi \otimes \widehat{(-i \xi) \psi}$ and noticing that
\[
\int_{\mathbb{R}} \partial_{x_3}w(x',x_3) \widehat{(-i \xi)\psi}(x_3) dx_3
= - \int_{\mathbb{R}} w(x',x_3) \widehat{(-i \xi)^2\psi}(x_3) dx_3
= - \int_{\mathbb{R}} \widehat{w}(x',\xi) (-i \xi)^2 \psi(\xi) d \xi,
\]
we find out that
\begin{equation}\label{39}
\int_\Omega C_{33}(x')\partial_{x_3}w(x',x_3)\partial _{x_3}(\varphi \otimes \widehat{\psi})(x',x_3)dx'dx_3=\int_\Omega C_{33}(x')(-i\xi )^2\widehat{w}(x',\xi )(\varphi \otimes \psi)(x',\xi )dx'd\xi.
\end{equation}
Finally, putting \eqref{35}-\eqref{39} together, we end up getting that 
\[
\langle -\mdiv_{x'} (\widetilde{C}(x')\nabla _{x'}\widehat{w})+P(x',\xi )\cdot \nabla _{x'}\widehat{w}+q(x',\xi )\widehat{w},\Phi \rangle = 0,\;\; \Phi \in C_0^\infty (\omega )\otimes C_0^\infty (\mathbb{R}),
\]
where $\langle \cdot ,\cdot \rangle$ denotes the duality pairing between $C_0^\infty (\Omega )$ and $\mathcal{D}'(\Omega )$. From this and the density of $C_0^\infty (\omega )\otimes C_0^\infty (\mathbb{R})$ in $C_0^\infty (\Omega )$ then follows that
\[
-\mdiv_{x'} (\widetilde{C}(x')\nabla _{x'}\widehat{w})+P(x',\xi )\cdot \nabla _{x'}\widehat{w}+q(x',\xi )\widehat{w}=0\;\; \textrm{in}\; \mathcal{D}'(\Omega ),
\]
which completes the proof.
\end{proof}

Let $a \in \mathbb{R}$ be fixed.  We assume in the remaining of this section that $\alpha (x_3)=a$ for all $x_3 \in \mathbb{R}$. For notational simplicity we write $A_a(x')$ instead of $A(x',\alpha (x_3))$. With the help of Lemma \ref{lemma3} we will first re-express \eqref{2} in the Fourier plane. 

For $g\in H^1(\mathbb{R})\cap L^1(\mathbb{R})$ such that $\int_\mathbb{R}g(x_3)dx_3=1$ and for
$h \in H^{1 \slash 2}(\partial \omega)$, we consider the $H^1(\Omega)$-solution $u$ to \eqref{2}, with
$$ f(x',x_3)=  g(x_3) h(x'),\ x' \in \partial \omega,\ x_3 \in \mathbb{R}. $$
Since $u$ is solution to \eqref{34} with $C=A_a \in W^{1,\infty}(\omega)^{3 \times 3}$, we deduce from Lemma \ref{lemma3} that
$\widehat{u}\in L^2(\mathbb{R};H^1(\omega ))$ is solution to the system
\begin{equation}\label{41}
\left\{
\begin{array}{ll}
-\mdiv_{x'}\big(\widetilde{A}_a(x')\nabla_{x'} \widehat{u}(x',\xi )\big)-2ia\xi x'{^\bot} \cdot \nabla _{x'}\widehat{u}+\xi ^2\widehat{u}=0\;\; &\textrm{in}\;\; \mathcal{D}'(\Omega ), 
\\
\widehat{u}(\cdot ,\xi)= \widehat{g}(\xi )f &\textrm{on}\;\; \partial \omega ,\; \textrm{for all}\; \xi \in \mathbb{R},
\end{array}
\right.
\end{equation}
where
\[
x'{^\bot}=(-x_2,x_1)\ {\rm and}\
\widetilde{A}_a(x')=
 \begin{pmatrix}
1+x_2^2a^2 & -x_2x_1a^2 
\\
-x_2x_1a^2 & 1+x_1^2a^2 
 \end{pmatrix}.
\]
We turn now to examining \eqref{41}.\\

\noindent {\bf Analysis of the variational problem associated with \eqref{41}.}
Let us consider the bilinear form
\begin{align*}
\mathcal{A}_\xi[(v,w),(\varphi ,\psi)]&= \int_\omega \widetilde{A}_a\nabla v\cdot \nabla \varphi dx' -2a\xi \int_\omega x'{^\bot}\cdot \nabla w \varphi dx'+\xi ^2\int_\omega v\varphi dx'
\\
&\; +\int_\omega \widetilde{A}_a\nabla w\cdot \nabla \psi dx' +2a\xi \int_\omega x'{^\bot}\cdot \nabla v \psi dx'+\xi ^2\int_\omega w\psi dx',\ (v,w),(\varphi ,\psi) \in \mathcal{H},
\end{align*}
defined on the Hilbert space $\mathcal{H}=H_0^1(\omega )\times H_0^1(\omega )$ endowed with the norm
$$ \|(v,w)\|_{\mathcal{H}}=\big(\| \nabla v\|_{L^2(\omega )}^2+\| \nabla w\|_{L^2(\omega )}^2\big)^{1/2}. $$
Taking into account that
\[
\widetilde{A}_a(x')\zeta \cdot \zeta \geq |\zeta |^2,\; \textrm{for all}\; \zeta \in \mathbb{R}^2\; \textrm{and}\; x'\in \omega,
\]
and that
$$2a|\xi |\int_\omega \Big|x'{^\bot} \cdot \nabla v w\Big|dx'\leq a^2\delta ^2\int_\omega |\nabla v|^2 dx'+\xi ^2\int_\omega w^2dx',\ (v,w) \in \mathcal{H}, $$ 
where $\delta=\max_{x' \in \omega} |x'| < \infty$, it is easy to see that
\begin{equation}\label{42}
\mathcal{A}_\xi[(v,w),(v,w)]\geq (1-a^2\delta ^2)\|(v,w)\|_{\mathcal{H}}^2.
\end{equation}
Let us fix $a_0>0$ so small that $\alpha=1-a_0^2\delta ^2>0$. In light of the above estimate, the bilinear form $\mathcal{A}_\xi$ is $\alpha$-elliptic for every $\xi \in \mathbb{R}$, provided we have $|a| \leq a_0$.
For each $\Phi \in C(\mathbb{R}; \mathcal{H}')$ and every $\xi \in \mathbb{R}$, there is thus a 
unique $(v(\xi ), w(\xi ))\in \mathcal{H}$ satisfying
\begin{equation}\label{43}
\mathcal{A}_\xi[(v(\xi),w(\xi)),(\varphi ,\psi )]=\langle \Phi (\xi ), (\varphi ,\psi)\rangle \;\; \textrm{for all}\;\; (\varphi ,\psi)\in \mathcal{H},
\end{equation}
by Lax-Milgram's lemma. From this then follows that
\begin{eqnarray}
& & \mathcal{A}_{\xi +\eta}[(v(\xi +\eta )-v(\xi ) ,w(\xi +\eta )-w(\xi )),(\varphi ,\psi )] \nonumber \\
&=& \mathcal{A}_\xi [(v(\xi ) ,w(\xi )),(\varphi ,\psi )]-\mathcal{A}_{\xi +\eta}[(v(\xi  ),w(\xi )),(\varphi ,\psi )]
+ \langle\Phi (\xi +\eta )-\Phi (\xi ), (\varphi,\psi) \rangle,
\label{44}
\end{eqnarray}
for each $\xi, \eta \in \mathbb{R}$ and $(\varphi ,\psi)\in \mathcal{H}$.
Further, by noticing through elementary computations that
\[
\mathcal{A}_{\xi+\eta}[(v,w),(\varphi ,\psi )]
=\mathcal{A}_{\xi+\eta}[(v,w),(\varphi ,\psi )]  - 2 a \eta \int_{\omega} x'^{\perp} \cdot ( \varphi \nabla_{x'} w - \psi \nabla_{x'} v) dx'+ \eta (2 \xi +\eta) \int_{\omega} (v \varphi + w \psi) dx', 
\]
for every $(v,w), (\varphi,\psi) \in \mathcal{H}$, we deduce from \eqref{44} and
Poincar\'e's inequality that there exists a constant $C=C(\xi ,\omega ,a_0)>0$ satisfying
\begin{eqnarray}
& & \mathcal{A}_\xi [(v(\xi ) ,w(\xi )),(v(\xi +\eta )-v(\xi ),w(\xi +\eta )-w(\xi ))]  \nonumber \\
& - & \mathcal{A}_{\xi +\eta}[(v(\xi  ),w(\xi )),(v(\xi +\eta )-v(\xi ) ,w(\xi +\eta )-w(\xi ))] \nonumber \\
& \leq & C |\eta | \|(v(\xi ) ,w(\xi ))\|_{\mathcal{H}} \|(v(\xi +\eta )-v(\xi ),w(\xi +\eta )-w(\xi ))\|_{\mathcal{H}},
\label{45}
\end{eqnarray}
for all $\xi \in \mathbb{R}$ and $\eta \in [-1,1]$.
In light of \eqref{42}, \eqref{44} written with $(\varphi ,\psi )=(v(\xi +\eta )-v(\xi ),w(\xi +\eta )-w(\xi ))$ and \eqref{45},
we thus find out that
\[
\alpha \|(v(\xi +\eta )-v(\xi ),w(\xi +\eta )-w(\xi ))\|_{\mathcal{H}}\leq C|\eta | \|(v(\xi ) ,w(\xi ))\|_{\mathcal{H}}+\|\Phi (\xi +\eta )-\Phi (\xi )\|_{\mathcal{H}'}.
\]
This proves that $(v,w)\in C(\mathbb{R};\mathcal{H})$. Moreover, we obtain
\begin{equation}\label{46}
\|(v(\xi ),w(\xi )) \|_{\mathcal{H}}\leq (1/\alpha )\|\Phi (\xi )\|_{\mathcal{H}'},\;\; \xi \in \mathbb{R},
\end{equation}
directly from \eqref{42}-\eqref{43}.
Further, it is easy to check for $\Phi \in C^1(\mathbb{R};\mathcal{H}')$ that $(v'(\xi ),w'(\xi )) \in C(\mathbb{R},\mathcal{H})$ is the solution to the variational problem
\[
\mathcal{A}_{\xi}[((v'(\xi ),w'(\xi )),(\varphi ,\psi)]=\langle \Phi_0(\xi ),(\varphi ,\psi)\rangle +\langle \Phi '(\xi ),(\varphi ,\psi)\rangle \;\; \textrm{for all}\;\; 
(\varphi ,\psi)\in \mathcal{H},
\]
where
\[
\langle \Phi_0(\xi ),(\varphi ,\psi)\rangle =2a \int_\omega x'{^\bot}\cdot ( \varphi \nabla_{x'} w(\xi )  - \psi \nabla_{x'} v(\xi) ) dx'-2 \xi \int_\omega (v(\xi )\varphi +  w(\xi )\psi) dx' .
\]
Using \eqref{46} and noting $\langle \xi \rangle=(1+|\xi |^2)^{1/2}$, we deduce from the above estimate that
\[
\| (v'(\xi ),w'(\xi ))\|_{\mathcal{H}}\leq C\big( \langle \xi \rangle \|\Phi (\xi )\|_{\mathcal{H}'}+\|\Phi '(\xi )\|_{\mathcal{H}'}\big),\;\; \xi \in \mathbb{R},
\]
for some constant $C=C(a_0,\omega )>0$. Similarly, if $\Phi \in C^2(\mathbb{R};\mathcal{H}')$ then the same reasoning shows that $(v,w)\in C^2(\mathbb{R},\mathcal{H})$ verifies
\[
\| (v''(\xi ),w''(\xi ))\|_{\mathcal{H}}\leq C\big(\langle \xi \rangle ^2\|\Phi (\xi )\|_{\mathcal{H}'}+ \langle \xi \rangle \|\Phi '(\xi )\|_{\mathcal{H}'}+\|\Phi ''(\xi )\|_{\mathcal{H}'}\big),\;\; \xi \in \mathbb{R}.
\]
Summing up we obtain the:
\begin{proposition}\label{proposition2}
Let $a_0 \in (0,(\max_{x' \in \omega} |x'|)^{-1 \slash 2})$, pick $a \in [-a_0,a_0]$ and assume that $\alpha(x_3)=a$ for all $x_3 \in \mathbb{R}$. Then for every $\Phi \in H^2(\mathbb{R},\mathcal{H}')$ such that $\langle \xi \rangle ^{2-j} \Phi^{(j)} \in L^2(\mathbb{R},\mathcal{H}')$, $j=0,1$, the variational problem \eqref{43} admits a unique solution $(v,w)\in H^2(\mathbb{R};\mathcal{H})$ satisfying
\begin{equation}\label{47}
\| (v,w)\|_{H^2(\mathbb{R};\mathcal{H})}\leq C\Big(\sum_{j=0}^2\| \langle \xi \rangle ^{2-j}\Phi ^{(j)}\|_{L^2(\mathbb{R};\mathcal{H}')}\Big),
\end{equation}
for some constant $C=C(\omega ,a_0)>0$. 
The above assumptions on $\Phi$ are actually satisfied whenever $\Phi =\widehat{\Psi}$ for some $\Psi \in H^2(\mathbb{R}; \mathcal{H}')$ such that $x_3\Psi \in H^1(\mathbb{R}; \mathcal{H}')$ and $x_3^2\Psi \in L^2(\mathbb{R}; \mathcal{H}')$. Moreover, the estimate \eqref{47} reads
\begin{equation}\label{48}
\| (v,w)\|_{H^2(\mathbb{R};\mathcal{H})}\leq C\Big(\sum_{j=0}^2\| x_3^{j}\Psi\|_{H^{2-j}(\mathbb{R};\mathcal{H}')}\Big),
\end{equation}
in this case.
\end{proposition}
Armed with Proposition \ref{proposition2} we may now tackle the analysis of the solution to \eqref{41}.\\

\noindent {\bf Some useful properties of the solution to \eqref{41}.} Pick $F\in H^1(\omega )$ such that $F=f$ on $\partial \omega$ and $\|F\|_{H^1(\omega )}\leq C(\omega )\|f\|_{H^{1/2}(\partial \omega )}$. Let $\widetilde{u}^r$ (resp., $\widetilde{u}^i$) denote the real (resp., imaginary) part of $\widetilde{u}=\widehat{u}-\widehat{g}(\xi )F=\widetilde{u}^r+i\widetilde{u}^i$. Since the Fourier transform $\widehat{u}$ of the $H^1(\Omega)$-solution $u$ to \eqref{2} is actually solution to \eqref{41}, we get by
direct calculation that $(\widetilde{u}^r,\widetilde{u}^i)$ is solution to the variational problem \eqref{43}, with
\begin{equation}
\label{48b}
\langle \Phi (\xi ), (\varphi ,\psi )\rangle =-\mathcal{A}_\xi [(\widehat{g}^rF,\widehat{g}^iF), (\varphi ,\psi )],
\end{equation}
and where $\widehat{g}^r$ (resp., $\widehat{g}^i$) stands for the real (resp., imaginary) part of $\widehat{g}$.
In light of \eqref{48b} we check out using elementary computations that
\begin{eqnarray}
\|\Phi (\xi )\|_{\mathcal{H}'} & \leq  & C\langle \xi \rangle ^2|\widehat{g}(\xi )|\|f\|_{H^{1/2}(\partial \omega )}
\label{48c1} \\
\|\Phi '(\xi )\|_{\mathcal{H}'} & \leq & C\big(\langle \xi \rangle ^2|\widehat{g}'(\xi )|+\langle \xi \rangle |\widehat{g}(\xi )|\big)|\|f\|_{H^{1/2}(\partial \omega )}
\label{48c2} \\
\|\Phi ''(\xi )\|_{\mathcal{H}'} & \leq & C\big(\langle \xi \rangle ^2|\widehat{g}''(\xi )|+\langle \xi \rangle |\widehat{g}'(\xi )|+|\widehat{g}(\xi )|\big)\|f\|_{H^{1/2}(\partial \omega )}, \label{48c3}
\end{eqnarray}
for some constant $C=C(a_0,\omega )>0$. Therefore we have $\langle \xi \rangle^j \Phi^{(2-j)} \in L^2(\mathbb{R})$ for $j=0,1,2$, provided $\langle \xi \rangle^{4-j} \widehat{g}^{j} \in L^2(\mathbb{R})$, which is actually the case if $x_3^j g \in H^{4-j}(\mathbb{R})$. From this and Proposition~\ref{proposition2} then follows the:
\begin{corollary}\label{corollary1}
Let $a$ and $\alpha$ be the same as in Proposition \ref{proposition2} and
assume that $g\in H^4(\mathbb{R})$ verifies $x_3g\in H^3(\mathbb{R})$, $x_3^2g\in H^2(\mathbb{R})$ and $\int_\mathbb{R}g(x_3)dx_3=1$. Then we have $\widehat{u}\in H^2(\mathbb{R};H^1(\omega ))$, $u\in L^1(\mathbb{R};H^1(\omega ))$ and $U=\widehat{u}(\cdot ,0)=\int_\mathbb{R}u(\cdot ,x_3)dx_3\in H^1(\omega )$ is the variational solution to the BVP
\[
\left\{
\begin{array}{ll}
\mdiv_{x'} \big(\widetilde{A}_a \nabla_{x'} U\big)=0\;\; &\textrm{in}\; \omega  
\\
U= f &\textrm{on}\; \partial \omega.
\end{array}
\right.
\]
\end{corollary}

In view of \eqref{41} and \eqref{43}, we deduce from \eqref{48c1}-\eqref{48c3} that
\begin{align*}
&\| \mdiv_{x'}\big(\widetilde{A}_a \nabla _{x'}\widehat{u}(\cdot ,\xi )\big) \|_{L^2(\omega)}\leq C\langle \xi \rangle ^4|\widehat{g}(\xi )|\|f\|_{H^{1/2}(\partial \omega )}
\\
&\|\partial _\xi\mdiv_{x'}\big(\widetilde{A}_a \nabla _{x'}\widehat{u}(\cdot ,\xi )\big) \|_{L^2(\omega)}\leq C\big(\langle \xi \rangle ^5|\widehat{g}'(\xi )|+\langle \xi \rangle ^4|\widehat{g}(\xi )|\big)|\|f\|_{H^{1/2}(\partial \omega )}
\\
&\|\partial _\xi ^2\mdiv_{x'}\big(\widetilde{A}_a \nabla_{x'} \widehat{u}(\cdot ,\xi )\big)\|_{L^2(\omega)}\leq C\big(\langle \xi \rangle ^6|\widehat{g}(\xi )|+\langle \xi \rangle ^5|\widehat{g}'(\xi )|+\langle \xi \rangle ^4|\widehat{g}''(\xi )|\big)\|f\|_{H^{1/2}(\partial \omega )},
\end{align*}
for some positive constant $C$ depending only on $a_0$ and $\omega$. This combined with Corollary \ref{corollary1} yields the:

\begin{proposition}\label{proposition3}
Let $a$ and $\alpha$ be as in Proposition \ref{proposition2}. Let $g\in H^6(\mathbb{R})$ verify $x_3g\in H^5(\mathbb{R})$, $x_3^2g\in H^4(\mathbb{R})$ and $\int_\mathbb{R}g(x_3)dx_3=1$. Then we have $\widetilde{A}_a \nabla _{x'} \widehat{u} \cdot \nu (x')\in H^2(\mathbb{R};H^{-1/2}(\partial \omega))$ and thus $A_a\nabla u\cdot \nu (x)\in L^1(\mathbb{R};H^{-1/2}(\partial \omega))$, with 
\[\widetilde{A}_a  \nabla_{x'} U\cdot \nu (x')=\widetilde{A}_a \nabla_{x'} \widehat{u}(\cdot ,0)\cdot \nu (x')=\int_\mathbb{R}A_a\nabla u(\cdot ,x_3)\cdot \nu (x') d x_3 \in H^{-1/2}(\partial \omega ).\]
\end{proposition}

In light of Proposition~\ref{proposition3} it is natural to define the two following DN maps:
\begin{align*}
&\Lambda _a : f\in H^{1/2}(\partial \omega )\mapsto \left[ x_3\mapsto A_a\nabla u(\cdot ,x_3)\cdot \nu (\cdot ,x_3)\right] \in L^1(\mathbb{R};H^{-1/2}(\partial \omega ))
\\
&\widetilde{\Lambda}_a:f\in H^{1/2}(\partial \omega )\mapsto \widetilde{A}_a \nabla_{x'}U\cdot \nu (x')\in H^{-1/2}(\partial \omega ).
\end{align*}
These two operators are bounded, and they satisfy the estimate
\begin{equation}\label{49}
\|\widetilde{\Lambda}_1-\widetilde{\Lambda}_2\|_{\mathscr{L}(H^{1/2}(\partial \omega ),H^{-1/2}(\partial \omega ))}\leq \|\Lambda _1-\Lambda _2\|_{\mathscr{L}\big(H^{1/2}(\partial \omega ),L^1(\mathbb{R};H^{-1/2}(\partial \omega ))\big)},
\end{equation}
where, for simplicity, we write $\Lambda _j$ (resp., $\widetilde{\Lambda}_j$) for $\Lambda _{a_j}$ (resp., $\widetilde{\Lambda}_{a_j}$), $j=1,2$.

Finally, since the matrix $\partial_a \widetilde{A}(x',a)$ has two eigenvalues $\lambda _0=0$ and $\lambda _1=|x'|^2$, we derive the following result by mimicking the proof of \cite[Claim, page 169]{AG1}.

\begin{theorem}\label{theorem2}
Let $a_0$ be the same as in Proposition~\ref{proposition2}, let $a_j \in \mathbb{R}$, $j=1,2$, and assume that $\alpha_j(x_3)=a_j$ for all $x_3 \in \mathbb{R}$.  Then there exists a constant $C>0$, depending only on $a_0$ and $\omega$, such that the following stability estimate
\[
|a_1-a_2|\leq C \|\Lambda _1-\Lambda _2\|_{\mathscr{L}\big(H^{1/2}(\partial \omega ),L^1(\mathbb{R};H^{-1/2}(\partial \omega ))\big)},
\]
holds true whenever $|a_1|,\ |a_2|\leq a_0$.
\end{theorem}


\appendix


\section{Restriction to $\widetilde{H}^{3/2}(\partial \Omega )$}
In this subsection we exhibit sufficient conditions on $\omega$ and $\alpha$ ensuring that the restriction of $\Lambda_\alpha$ to $\widetilde{H}^{3/2}(\partial \Omega )$ is a bounded operator into $L^2(\mathbb{R}; H^{1/2}(\partial \omega ))$. We assume for this purpose that $\Omega_1= \omega \times (-1,1)$ has $H^2$-regularity property. That is,
for every $F\in L^2(\Omega )$ and any matrix-valued function $C=(C_{ij}(x))_{1 \leq i,j \leq 3}$ with coefficients in $W^{1,\infty}(\Omega _1)$
verifying the ellipticity condition 
\[
\exists \alpha >0,\ C(x)\xi \cdot \xi \geq \alpha |\xi |^2,\ \textrm{for all}\; \xi \in \mathbb{R}^3,\; x\in \Omega_1,
\]
the following BVP
 \[
\left\{
\begin{array}{ll}
\mdiv (C\nabla w)=F\;\; &\textrm{in}\;\; \Omega _1,
\\
w=0 &\textrm{on}\;\; \partial \Omega _1,
\end{array}
\right.
\]
has a unique solution $w\in H^2(\Omega _1)$ obeying
\[
\|w\|_{H^2(\Omega _1)}\leq C(\alpha , M)\|F\|_{L^2(\Omega _1)},
\]
for some constant $C(\alpha ,M)>0$ depending only on $\alpha$, $M=\max_{1 \leq i,j \leq 3}\|C_{ij}\|_{W^{1,\infty}(\Omega _1)}$ and $\omega$.

Notice that $\Omega_1$ has $H^2$-regularity property if and only if this is the case for 
$\Omega _a=\omega \times (-a,a)$ and some $a>0$. Moreover we recall from \cite{G} that $\Omega_1$ has $H^2$-regularity property provided $\omega$ is convex.

We turn now to establishing the following result, which is our main tool for the analysis of the restriction of $\Lambda_\alpha$ to $\tilde{H}^{3 \slash 2}(\partial \Omega )$.
\begin{theorem}\label{theorem1}
Assume that $\alpha \in C^{0,1}(\mathbb{R})$ and that $\Omega _1$ has $H^2$-regularity property.
Then for any $f\in \widetilde{H}^{3/2}(\partial \Omega )$, the BVP \eqref{2} admits a unique solution $u\in H^2(\Omega )$. Moreover there is a constant $C>0$, depending only on $\|\alpha \|_{C^{0,1}(\mathbb{R})}$ and $\omega$, such that we have:
\begin{equation}\label{17}
\|u\|_{H^2(\Omega )}\leq C \|f\|_{\widetilde{H}^{3/2}(\partial \Omega )}.
\end{equation}
\end{theorem}
\begin{proof}
Since $f\in \widetilde{H}^{3/2}(\partial \Omega )$ we may choose $F\in H^2(\Omega )$ in accordance with Lemma~\ref{lemma1bis} so that $F=f$ on $\partial \Omega$ and
\begin{equation}\label{18}
\|F\|_{H^2(\Omega )} = \|f\|_{\widetilde{H}^{3/2}(\partial \Omega )}.
\end{equation}
We put $\Psi =\mdiv (A\nabla F)$. By the ellipticity condition \eqref{8} we have a unique $u_0\in H_0^1(\Omega )$ satisfying simultaneously
\begin{equation}\label{19}
\int_\Omega A\nabla u_0\cdot \nabla v dx=\int_\Omega \Psi v dx,\; \textrm{for all}\; v\in H_0^1(\Omega ),
\end{equation}
and
\begin{equation}\label{20}
\|u_0\|_{H^1(\Omega )}\leq C_0 \|\Psi \|_{L^2(\Omega )},
\end{equation}
for some constant $C_0>0$ depending on $\omega$ and $M=\|\alpha \|_{C^{0,1}(\mathbb{R})}$.

Further, there exists $\xi _n\in C_0^\infty (-(n+1),n+1)$ such that $\xi _n=1$ in $[-n,n]$, $\|\xi_n'\|_{\infty} \leq \kappa$ and $\|\xi_n'' \|_{\infty} \leq \kappa$ for all $n \geq 1$, by \cite[Theorem 1.4.1 and Eq. (1.4.2)]{Ho} and estimate (1.4.2) \cite[p. 25]{Ho}, where $\kappa$ is a constant which is independent of $n$. 

Thus, we get for any $v\in H_0^1(\Omega )$ and $n \geq 1$ that
\[
\int_\Omega A\nabla (\xi _n u_0)\cdot \nabla v dx=\int_\Omega A\nabla u_0\cdot \nabla (\xi _nv) dx-\int_\Omega (A\nabla u_0\cdot \nabla \xi _n) v dx+\int_\Omega (A\nabla \xi _n \cdot \nabla v) u_0 dx,
\]
by direct calculation.
An integration by parts in the last term of this identity providing
\[
\int_\Omega (A\nabla \xi _n \cdot \nabla v) u_0 dx= -\int_\Omega (A\nabla \xi _n \cdot \nabla u_0) v dx-\int_\Omega \mdiv (A\nabla \xi_n) u_0 v dx,
\]
we find out that
\[
\int_\Omega A\nabla (\xi _n u_0)\cdot \nabla v dx=\int_\Omega A\nabla u_0\cdot \nabla (\xi _nv) dx -\int_\Omega ( A\nabla u_0\cdot \nabla \xi _n ) v dx -\int_\Omega (A\nabla \xi _n \cdot \nabla u_0) v dx-\int_\Omega \mdiv (A\nabla \xi _n)u_0 v dx.
\]
Since $A$ is symmetric, it follows from this and \eqref{19} that
\[
\int_\Omega A\nabla (\xi _n u_0)\cdot \nabla v dx=\int_\Omega \Psi \xi _nv dx-2\int_\Omega ( A\nabla \xi _n \cdot \nabla u_0) v dx-\int_\Omega \mdiv (A\nabla \xi _n)u_0 v dx,\ \textrm{for all}\ v \in H_0^1(\Omega ).
\]
Therefore, bearing in mind that $\Omega_a=\omega \times (-a,a)$ for any $a>0$, the $\xi_n u_0 \in H_0^1(\Omega _{n+1})$ is thus solution to the variational problem
\begin{equation}\label{21}
\int_{\Omega _{n+1}}A\nabla (\xi _nu_0)\cdot \nabla v dx=\int_{\Omega _{n+1}}\widetilde{\Psi}v dx,\; \textrm{for all}\; v\in H_0^1(\Omega _{n+1}),
\end{equation}
with 
\[
\widetilde{\Psi}= \Psi \xi _n-2A\nabla \xi _n \cdot \nabla u_0-\mdiv (A\nabla \xi _n)u_0.
\]
The next step of the proof is to make the change of variables $(x',x_3)\in \Omega _{n+1}\mapsto (x', y_3)=(x',1/(n+1)x_3) \in \Omega _1$ in \eqref{21}. Putting
$$ J_n= \left( \begin{array}{ccc} 1 & 0 & 0 \\ 0 & 1 & 0 \\ 0 & 0 & 1 \slash n \end{array} \right),\ n \geq 1, $$
and
\begin{align*}
&\underline{A}(x',y_3)=1/(n+1)J_{n+1}A(x', (n+1)y_3)J_{n+1},
\\
&\underline{\xi}(y_3)=\xi _n((n+1)y_3),
\\
& \underline{u}(x',y_3)=u_0(x',(n+1)y_3),
\\
& w_n(x',y_3)=\xi _n((n+1)y_3)u_0(x',(n+1)y_3),
\\
& \underline{\mdiv} (P(x',y_3))= \partial _{x_1}P_1(x',y_3)+\partial _{x_2}P_2(x',y_3)+1/(n+1)\partial _{y_3}P_3(x',y_3),
\\
& \underline{\Psi}(x',y_3)=1/(n+1)\Big[\Psi (x',(n+1)y_3)-2J_{n+1}A(x', (n+1)y_3)J_{n+1}\nabla \underline{\xi}(y_3)\cdot \nabla \underline{u}(x',y_3)
\\
& \hskip 3.5cm -\underline{\mdiv} \big( A(x', (n+1)y_3) J_{n+1} \nabla \underline{\xi}(y_3) \big)\underline{u}(x',y_3)\Big],
\end{align*}
for $(x',y_3) \in \Omega_1$, we find out by direct computations that $w_n \in H_0^1(\Omega _1)$ is solution to
\[
\int_{\Omega _1}\underline{A}\nabla w_n\cdot \nabla v dx=\int_{\Omega _1}\underline{\Psi}v dx,\; \textrm{for all}\; v\in H_0^1(\Omega _1).
\]
Since $\|\underline{\Psi}\|_{L^2(\Omega _1)}\leq (n+1)^{-3/2}C(\omega,M )\|\Psi \|_{L^2(\Omega )}$ by \eqref{20}, where $C=C(\omega,M)$ denotes some generic positive constant depending only on $\omega$ and $M$, it holds true that
$\|w_n\|_{H^2(\Omega _1)}\leq (n+1)^{-3/2}C(\omega, M )\|\Psi \|_{L^2(\Omega )}$.
Here we used the estimate $\|w_n\|_{H^2(\Omega _1)}\leq C(\omega, M )\|\underline{\Psi}\|_{L^2(\Omega _1)}$, arising from the $H^2$-regularity property imposed on $\Omega_1$.
As a consequence we have
 \begin{equation}\label{23}
\|\xi _nu_0\|_{H^2(\Omega )}\leq C(\omega,M )\|\Psi \|_{L^2(\Omega )}.
\end{equation}
Therefore, upon eventually extracting a subsequence of $(\xi _n u_0)_n$, we may assume that it converges weakly to $\widetilde{u}$ in $H^2(\Omega )$. On the other hand $(\xi _n u_0)_n$ converges to $u_0$ in $L^2(\Omega )$. Thus by the uniqueness of the limit, we have $u_0=\widetilde{u}\in H^2(\Omega )$ so $(\xi _nu_0)_n$ converges weakly to $u_0$ in $H^2(\Omega )$. Further, the norm $\|\cdot \|_{H^2(\Omega )}$ being lower semi-continuous, we have 
\begin{equation}
\label{23b}
\|u_0\|_{H^2(\Omega )}\leq \liminf_n \|\xi _nu_0\|_{H^2(\Omega )}\leq C(\omega,M )\|\Psi \|_{L^2(\Omega )},
\end{equation}
by \eqref{23}.
Bearing in mind that $\|\Psi \|_{L^2(\Omega )}\leq C(\omega,M )\|F\|_{H^2(\Omega )}$, \eqref{18} and \eqref{23b} then yield
\[
\|u_0\|_{H^2(\Omega )}\leq C(\omega, M )\|f\|_{\widetilde{H}^{3/2}(\partial \Omega )}.
\]
Now the desired result follows from this by taking into account that $u=u_0+F\in H^2(\Omega )$ is the unique solution to \eqref{2}.
\end{proof}

For all $\Omega$ and $\alpha$ fulfilling the assumptions of Theorem~\ref{theorem1}, the mapping
\[
\Lambda _\alpha : f\in \widetilde{H}^{3/2}(\partial \Omega ) \mapsto \partial _\nu u\in L^2(\mathbb{R}; H^{1/2}(\partial \omega )),
\]
where $u$ denotes the unique $H^2(\Omega)$-solution to \eqref{2}, is well defined by Theorem~\ref{theorem1}.
Further, since $C_0^\infty (\mathbb{R};H^2(\omega ))$ is dense in $H^2(\Omega )$, then the trace operator
\[
\widetilde{\tau} : w\in H^2(\Omega )\mapsto \partial _\nu w\in L^2(\mathbb{R}; H^{1/2}(\partial \omega )),
\]
is easily seen to be bounded. From this and
\eqref{17} then follows that $\|\Lambda _\alpha \|\leq C$ as a linear bounded operator from
$\widetilde{H}^{3/2}(\partial \Omega )$ into $L^2(\mathbb{R}; H^{1/2}(\partial \omega ))$, where the constant $C>0$ depends only on $\omega$ and $\| \alpha \|_{C^{0,1}(\mathbb{R})}$.


\section{Linking $\Lambda_\alpha$ to the DN map associated with $\Omega_{\theta}$}
\label{sec-DNO}
In this subsection we define the DN map $\widetilde{\Lambda}_\theta$ associated with the BVP \eqref{1}, which is stated on the twisted domain $\Omega_\theta$, and establish the link between
$\widetilde{\Lambda}_\theta$ and $\Lambda_\alpha$, where we recall that $\alpha=\theta'$. It turns out that $\widetilde{\Lambda}_\theta$ is not physically relevant\footnote{Since $\widetilde{\Lambda}_\theta$ is defined from the Neumann observation on $\partial \Omega_\theta$ of the solution to \eqref{1} then the variable twisting angle $\theta \in C^1(\mathbb{R})$ should necessarily be known everywhere.} for the analysis of the inverse problem under consideration in this text, but since the BVP \eqref{2} was derived from \eqref{1}, it is plainly natural to link the operator $\Lambda_\alpha$ we used in the preceding sections to $\widetilde{\Lambda}_\theta$.

We start by defining the trace space for functions in $H^1(\Omega _\theta )$. We set for all $L>0$,
\[
\Omega _\theta ^L=\{ (R_{\theta (x_3)}x',x_3);\; x'=(x_1,x_2)\in \omega ,\; x_3\in (-L,L)\} = \{ x \in \Omega_\theta,\ | x_3 | < L \}
\]
and
\[
\Gamma _\theta ^L=\{ (R_{\theta (x_3)}x',x_3);\; x'=(x_1,x_2)\in \partial \omega ,\; x_3\in [-L,L]\} = \{ x \in \partial \Omega_\theta,\ | x_3 | < L \}.
\]
For every $u\in H^1(\Omega _\theta )$ we have $u_{|\Omega _\theta ^L} \in H^1(\Omega _\theta ^L)$ hence $u_{|\partial \Omega _\theta ^L}\in H^{1/2}(\partial \Omega _\theta ^L)$. Thus, putting
\[
H^{1/2}(\Gamma _\theta ^L)=\{ h=g_{|\Gamma _\theta ^L}\ \textrm{in}\ L^2(\Gamma _\theta ^L);\ g\in H^{1/2}(\partial \Omega _\theta ^L) \},
\]
it holds true that $u_{| \Gamma_\theta^L} \in H^{1/2}(\Gamma _\theta ^L)$. Here the space $H^{1/2}(\Gamma _\theta ^L)$ is equipped with its natural quotient norm 
$$\|h\|_{H^{1/2}(\Gamma _\theta ^L)}=\inf \{ \|g\|_{H^{1/2}(\partial \Omega _\theta ^L)};\; g_{|\Gamma _\theta ^L}=h \}. $$ 
Further we define
\[
H_{\textrm{loc}}^{1/2}(\partial \Omega _\theta )=\{h\in L^2_{\textrm{loc}}(\partial \Omega _\theta );\;\; h_{|\Gamma _\theta ^L}\in H^{1/2}(\Gamma _\theta ^L)\; \textrm{for all}\ L>0 \},
\]
and then introduce the subspace
$\widetilde{H}^{1/2}(\partial \Omega _\theta )=\{ h\in H_{\textrm{loc}}^{1/2}(\partial \Omega _\theta );\; \textrm{there exists}\; v\in H^1(\Omega _\theta )\; \textrm{such that}\; v_{|\partial \Omega _\theta}=h\}$ of $H_{\textrm{loc}}^{1/2}(\partial \Omega _\theta )$.
Here and henceforth $v_{|\partial \Omega _\theta}=h$ means that the identity $v_{|\Gamma _\theta ^L}=h_{|\Gamma _\theta ^L}$ holds in the trace sense for every $L>0$.
It is not hard to see that $\widetilde{H}^{1/2}(\partial \Omega _\theta )$ is a Banach space for the quotient norm:
\[
\| h\|_{\widetilde{H}^{1/2}(\partial \Omega _\theta )}=\inf\{\|v\|_{H^1(\Omega _\theta )};\; v_{|\partial \Omega _\theta}=h\}.
\]

We introduce the mapping
\begin{align*}
I_\theta : C_0^1(\partial \Omega _\theta )&\longrightarrow C_0^1(\partial \Omega )
\\
g &\mapsto  f=g\circ \varphi _\theta ,
\end{align*}
where, for the sake of shortness, we note
$\varphi _\theta (x)=T_{\theta (x_3)}(x',x_3)$ for $x \in \overline{\Omega}$.
Pick $g$ in $C_0^1(\partial \Omega_\theta )$ and choose $v\in C_0^1(\mathbb{R}^3)$ such that $v_{|\partial \Omega _\theta}=g$. Since $v_{| \Omega_{\theta}} \in H^1(\Omega_{\theta})$ we get that $g \in \tilde{H}^{1 \slash 2}(\partial \Omega_\theta)$. Moreover for all $v \in H^1(\Omega_{\theta})$ obeying $v_{|\partial \Omega _\theta}=g$ the function
$u=v_{|\Omega _\theta}\circ \varphi _\theta$ belongs to $H^1(\Omega)$ by \cite[Proposition~9.6]{Bre}, and
\[
\|I_\theta g\|_{\widetilde{H}^{1/2}(\partial \Omega )}\leq \|u\|_{H^1(\Omega )}\leq C(\omega ,\theta )\| v\|_{H^1(\Omega _\theta )}.
\]
As a consequence we have
\begin{equation}\label{50}
\|I_\theta g\|_{\widetilde{H}^{1/2}(\partial \Omega )}\leq C(\omega ,\theta )\|g\|_{\widetilde{H}^{1/2}(\partial \Omega _\theta )}\;\; \textrm{for any}\; g\in C_0^1(\partial \Omega _\theta ).
\end{equation}

Let us now consider $g\in \widetilde{H}^{1/2}(\partial \Omega _\theta )$ and $v\in H^1(\Omega _\theta )$ such that $v_{|\partial \Omega _\theta}=g$. For any sequence $(v_n)_n \in C_0^1(\mathbb{R}^3)$ such that $v_n{_{|\Omega _\theta}} \mapsto v$ in $H^1(\Omega _\theta)$ as $n \rightarrow +\infty$, it is clear that
\[
\| g-g_n\|_{\widetilde{H}^{1/2}(\partial \Omega _\theta )}\leq \|v-v_n\|_{H^1(\Omega _\theta )},
\]
provided $g_n={v_n}_{|\partial \Omega _\theta}$. Hence $(g_n)_n$ converges to $g$ in $\widetilde{H}^{1/2}(\partial \Omega _\theta )$.

For all $n \geq 1$, put $f_n=I_\theta g_n=g_n\circ \varphi _\theta$ and $u_n=v_n\circ \varphi _\theta$. Since $f_n=u_n{_{|\partial \Omega}}$, we see that
\[
\|f_n-f_m\|_{\widetilde{H}^{1/2}(\partial \Omega )}\leq \|u_n-u_m\|_{H^1(\Omega )}\leq C(\omega ,\theta )\|v_n-v_m\|_{H^1(\Omega _\theta)}.
\]
Therefore $(f_n)_n$ is a Cauchy sequence in $\widetilde{H}^{1/2}(\partial \Omega )$ and $f=\lim_n f_n \in \widetilde{H}^{1/2}(\partial \Omega )$. Set $f=I_\theta g$. Then, in view of \eqref{50}, $I_\theta$ extends to a bounded operator, still denoted by $I_\theta$, from $\widetilde{H}^{1/2}(\partial \Omega _\theta )$ into $\widetilde{H}^{1/2}(\partial \Omega )$.

\smallskip
Arguing as above, we thus find out that the mapping
\begin{align*}
J_\theta : C_0^1(\partial \Omega  )&\longrightarrow C_0^1(\partial \Omega _\theta)
\\
f &\mapsto  g=f\circ \psi _\theta ,
\end{align*}
where $\psi_{\theta}=\varphi_{\theta}^{-1}$,
extends to a bounded operator, which is still called $J_\theta$, from $\widetilde{H}^{1/2}(\partial \Omega )$ into $\widetilde{H}^{1/2}(\partial \Omega _\theta )$.

\smallskip
Evidently, $I_\theta J_\theta f=f$ for all $f\in C_0^1(\partial \Omega  )$ and $J_\theta I_\theta g=g$ for all $g\in C_0^1(\partial \Omega _\theta)$. Therefore we have $J_\theta =I_\theta ^{-1}$ from the density of $C_0^1(\partial \Omega)$ (resp., $C_0^1(\partial \Omega _\theta)$) in $\tilde{H}^{1 \slash 2}(\partial \Omega)$ (resp., $\tilde{H}^{1 \slash 2}(\partial \Omega_\theta)$).

\smallskip
Next, by reasoning in the same way as in the derivation of \eqref{2}, we prove with the help of the Lax-Milgram lemma that the BVP \eqref{1} has a unique solution $v\in H^1(\Omega _\theta )$ for every $g\in \widetilde{H}^{1/2}(\partial \Omega _\theta )$. Moreover the operator 
\[ \widetilde{\Lambda}_\theta : g\mapsto \partial _\nu v\] 
is well defined as a bounded operator from $\widetilde{H}^{1/2}(\partial \Omega _\theta )$ into its dual space $\widetilde{H}^{-1/2}(\partial \Omega _\theta )$. Similarly to $\Lambda _\theta$, it can be checked that $\widetilde{\Lambda}_\theta$ is characterized by the following identity
\[
\langle \widetilde{\Lambda}_\theta g,h\rangle =\int_{\Omega _\theta}\nabla v \cdot \nabla H dy,
\]
which holds true for all $h\in \widetilde{H}^{1/2}(\partial \Omega _\theta )$ and all $H\in H^1(\Omega _\theta )$ such that $H_{|\partial \Omega _\theta}=h$. By performing the change of variable $y=\varphi _\theta (x)$ in the last integral, we thus get that
\[
\langle \widetilde{\Lambda}_\theta g,h\rangle =\int_\Omega A\nabla u \cdot \nabla (H\circ \varphi _\theta ) dx,
\]
where $u$ denotes the solution to the BVP \eqref{2} with $f=I_\theta g$. Therefore we have
\[
\langle \widetilde{\Lambda}_\theta g,h\rangle =\langle \Lambda _{\theta '}I_\theta g, I_\theta h \rangle,
\]
which means that $\widetilde{\Lambda}_\theta= I_\theta ^* \Lambda _{\theta '} I_\theta$, or equivalently that $\Lambda _{\theta '}=J_\theta ^* \widetilde{\Lambda}_\theta J_\theta$, where 
$\Lambda_{\theta'}$ stands for $\Lambda_\alpha$.


\bigskip

\end{document}